\documentclass{amsart}
\usepackage{amsmath}
\usepackage{amssymb}
\usepackage{amsfonts}
\usepackage{graphicx}
\usepackage{texdraw}
\usepackage{graphpap}
\usepackage{wasysym}
\usepackage{pxfonts}
\usepackage[OT2,T1]{fontenc}

\DeclareSymbolFont{cyrletters}{OT2}{wncyr}{m}{n}
\DeclareMathSymbol{\berd}{\beta}{cyrletters}{"42}

\input txdtools

 \newtheorem{thm}{Theorem}[section]
 \newtheorem{cor}[thm]{Corollary}
 \newtheorem{lem}[thm]{Lemma}
 \newtheorem{prop}[thm]{Proposition}
 \theoremstyle{definition}
 \newtheorem{defn}[thm]{Definition}

 \theoremstyle{remark}
 
 \newtheorem{rem}[thm]{Remark}

\numberwithin{equation}{section} \numberwithin{figure}{section}

\newcommand{\proj}{{P}}

\renewcommand{\d}{{\partial}}
\newcommand{\dbar}{{\overline{\partial}}}
\newcommand{\e}{\mathrm e}
\newcommand{\wh}{\widehat}
\newcommand{\dprime}{{\prime\prime}}
\newcommand{\1}{{\mathbf 1}}

\newcommand{\C}{{\mathbb C}}
\newcommand{\D}{{\mathbb D}}

\newcommand{\parT}{\rho}
\newcommand{\T}{{\mathbb T}}
\newcommand{\R}{{\mathbb R}}

\newcommand{\Z}{{\mathbb Z}}

\newcommand\coity{{\mathcal C \sb 0 \sp \infty}}

\newcommand{\esssup}{\operatorname{\text{\rm ess\,sup}\,}}
\newcommand{\essinf}{\operatorname{\text{\rm ess\,inf}\,}}

\newcommand{\calH}{{\mathcal H}}

\newcommand{\calC}{{\mathcal C}}

\newcommand{\calO}{{\mathcal O}}
\newcommand{\dA}{{\diff A}}

\newcommand{\calP}{{\mathcal P}}

\newcommand{\setS}{{\mathcal S}}
\newcommand{\setSp}{{\Sigma}}
\newcommand{\setX}{{X}}

\newcommand{\adiag}{{\overline{\text{\rm diag}}}}
\newcommand{\bfun}{{b}}
\newcommand{\cauchy}{C}
\newcommand{\berez}{\mathfrak B}
\newcommand{\diffP}{\diff P}
\newcommand{\Iop}{I}
\newcommand{\Pop}{P}
\newcommand{\Dop}{\nabla}

\newcommand\bpar{{M_1}}

\newcommand{\re}{\operatorname{Re}}

\newcommand{\wt}{\widetilde}

\newcommand{\pspace}{{\mathcal P}}

\newcommand{\Hspace}{{H}}

\newcommand{\diff}{{\mathrm d}}
\newcommand{\imag}{{\mathrm i}}

\newcommand{\dist}{\operatorname{dist\,}}
\newcommand{\supp}{\operatorname{supp}}

\newcommand{\pv}{\operatorname{p.v.}}

\newcommand{\SH}{{\rm SH}}
\newcommand{\eps}{\varepsilon}

\newcommand{\Tfun}{{\varrho}}
\newcommand{\Qext}{\psi}
\newcommand{\apar}{{M_0}}

\def\lpar{\left (}
\def\rpar{\right )}
\def\labs{\left |}
\def\rabs{\right |}
\def\babs#1{\labs {#1} \rabs}

\begin{document}
%
\title[Polynomial Berezin transform]
{Berezin transform in polynomial Bergman spaces}

\author[Ameur]
{Yacin Ameur}
\address{Yacin Ameur\\ Department of Mathematics\\
Uppsala University\\
\\
Box 480\\
751 06 Uppsala\\
Sweden}

\email{yacin.ameur@gmail.com}

\author[Hedenmalm]
{H\aa{}kan Hedenmalm}

\address{H{\aa}kan Hedenmalm\\ Department of Mathematics\\
The Royal Institute of Technology\\
S -- 100 44 Stockholm\\
SWEDEN}

\email{haakanh@math.kth.se}

\thanks{Research supported by the G\"oran Gustafsson Foundation. The third author is supported by N.S.F. Grant No. 0201893.}

\author[Makarov]
{Nikolai Makarov}

\address{Nikolai Makarov\\ Department of Mathematics\\ California Institute of
Technology\\ Pasadena\\ CA 91125\\ USA}

\email{makarov@caltech.edu}



\begin{abstract} We consider fairly general weight functions
$Q:\C\to \R$, and let
$K_{m,n}$ denote the reproducing kernel for the space
$\Hspace_{m,n}$ of analytic polynomials $u$ of degree at most $n-1$
of finite norm $\|u\|_{mQ}^2=\int_\C\babs{u(z)}^2\e^{-mQ(z)}\dA(z),$
$\dA$ denoting suitably normalized area measure in $\C$. For a
continuous bounded function $f$ on $\C$, we consider its
(polynomial) Berezin transform
\begin{equation*}\berez_{m,n}f(z)=\int_\C f(\zeta)\diff
B^{\langle z\rangle}_{m,n}(\zeta) \qquad\text{where}\qquad \diff
B^{\langle z\rangle}_{m,n}(\zeta)=\frac {\babs{K_{m,n}(z,\zeta)}^2}
{K_{m,n}(z,z)}\e^{-mQ(\zeta)}\dA(\zeta).
\end{equation*}
Let $\setX=\{\Delta Q>0\}.$ For a parameter $\tau>0$ we prove that
there exists a compact subset $\setS_\tau$ of $\C$ such that
\begin{equation}\label{berconv}\berez_{m,n}f(z)\to f(z)\qquad \text{as}
\quad m\to\infty\quad \text{and}\quad n- m\tau\to 0,\end{equation}
for all continuous bounded $f$ if $z$ is in the interior of
$\setS_\tau\cap \setX$. Equivalently, the measures $B_{m,n}^{\langle
z\rangle}$ converge to the Dirac measure at $z$. The set
$\setS_\tau$ is the coincidence set for an associated obstacle
problem.

We also prove that the convergence in \eqref{berconv} is
\textit{Gaussian} when $z$ is in the interior of
$\setS_\tau\cap\setX$, in the sense that with $F(\zeta)=f\lpar
\sqrt{m\Delta Q(z)}(\zeta-z)\rpar$, we have
\begin{equation*}
\berez_{m,n}F(z)\to \int_\C f(\zeta)\diffP(\zeta),
\end{equation*}
$\diffP(\zeta)=\e^{-\babs{\zeta}^2}\diff A(\zeta)$ denoting the
standard Gaussian.

In the "model case'' $Q(z)=\babs{z}^2$, $\setS_\tau$ is the closed
disk with centre $0$ and radius $\sqrt{\tau}$. We prove that if $z$
is fixed with $\babs{z}>\sqrt{\tau}$, the corresponding measures
$B_{m,n}^{\langle z\rangle}$ converge to harmonic measure for $z$
relative to the domain $\C^*\setminus\setS_\tau$, $\C^*$ denoting
the extended plane.

Our auxiliary results include $L^2$ estimates for the
$\dbar$-equation $\dbar u=f$ when $f$ is a suitable test function
and the solution $u$ is restricted by a growth constraint near
$\infty$.

Our results have applications e.g. to the study of weighted
distributions of eigenvalues of random normal matrices. In the companion paper \cite{AHM} we consider such applications, e.g. a proof of Gaussian field convergence
for fluctuations of linear statistics of eigenvalues of random normal matrices from the ensemble associated with $Q$.
\end{abstract}

\maketitle

\addtolength{\textheight}{2.2cm}

\section{Introduction} \label{intro}

\subsection{General introduction to Berezin quantization.}
\footnote{In this version we have but slightly modified the formulation of Theorem 2.8 and made some minor corrections.}

In a version of quantum theory, a Bargmann-Fock type space of
polynomials plays the role of the quantized system, while the
corresponding weighted $L^2$ space is the classical analogue. It is
therefore natural to study the asymptotics of the quantized system
as we approach the semiclassical limit. A particularly useful object
is the {\em reproducing kernel} of the Bargmann-Fock type space. To
make matters more concrete, let $\mu$ be a finite positive Borel
measure on $\C$, and $L^2(\C;\mu)$ the usual $L^2$ space with inner
product
$$\big\langle f,g\big\rangle_{L^2(\C;\mu)}=\int_\C f(z)\overline{g(z)}
\diff\mu(z).$$ The subspace of $L^2(\C;\mu)$ of entire functions is
the Bergman space $A^2(\C;\mu)$.

We assume that the support of $\mu$ has infinitely many points so
that any given polynomial corresponds to a unique element of
$L^2(\C;\mu)$, and write
$$J_\mu=\sup\big\{j\in\Z;\,
\int_\C\babs{z}^{2(j-1)}\diff\mu(z)<\infty\big\}.$$ Since $\mu$ is
finite, we are ensured that $1\le J_\mu\le \infty$.

Let $\pspace_n$ be the set of analytic polynomials of degree at most
$n-1$, and write
\begin{equation*}A^2_{\mu,n}=L^2(\C;\mu)\cap
\pspace_n\subset L^2(\C;\mu).\end{equation*} Putting
$n'=\min\{n,J_\mu\}$,
it is thus seen that $A^2_{\mu,n}$ equals $\pspace_{n'}$ in the
sense of sets, with the norm inherited from $L^2(\C;\mu)$. Hence
$A^2_{\mu,n}=A^2_{\mu,n'}$ for all $n$, and the reader who so
desires can without loss of generality assume that $n=n'$ in the
following.

Let $e_1,\ldots, e_{n'}$ be an orthonormal basis for $A^2_{\mu,n}$.
The reproducing kernel $K_{\mu,n}$ for the space $A^{2}_{\mu,n}$ is
the function
\begin{equation*}
K_{\mu,n}(z,w)=\sum_{j=1}^{n'}e_j(z)\overline{e_j(w)}.
\end{equation*}
Then $K_{\mu,n}$ is independent of the choice of
an orthonormal basis
for $A^2_{\mu,n}$, and it is characterized by the
properties that $z\mapsto K_{\mu,n}(z,z_0)$ is in $A^2_{\mu,n}$ and
\begin{equation}\label{reprod}u(z_0)=
\big\langle u,K_{\mu,n}(\cdot,z_0)\big\rangle_{L^2(\C;\mu)} =\int_\C
u(z) \overline{K_{\mu,n}(z,z_0)} \diff \mu(z), \qquad u\in
A^2_{\mu,n},\quad z_0\in\C.\end{equation}

For a given complex number $z_0$ we now consider the measure
\begin{equation}\label{bm}\diff
B^{\langle z_0\rangle}_{\mu,n}(z)=\frac {\babs{K_{\mu,n}(z,z_0)}^2}
{K_{\mu,n}(z_0,z_0)}\diff\mu(z),\end{equation} which we may call the
\textit{Berezin measure} associated with $\mu$, $n$ and $z_0$. The
reproducing property \eqref{reprod} applied to
$u=K_{\mu,n}(\cdot,z_0)$ implies that
$B_{\mu,n}^{\langle z_0\rangle}$ is a probability measure. In
classical physics,
$B^{\langle z_0\rangle}_{\mu,n}$ corresponds to a
point mass at $z_0$, so our interest focuses on how closely this
measure approximates the point mass.

\subsection{Weights.} By a \textit{weight} we mean a
measurable function $\phi:\C\to \R$ such that the measure
\begin{equation*}\diff\mu_\phi(z)=\e^{-\phi(z)}\diff A(z),\end{equation*}
is a finite measure on $\C$, where $\diff A(z)=\diff x\diff y/\pi$
with $z=x+\imag y$.

We write $L^2_\phi$ for the space $L^2(\C;\mu_\phi)$ and $A^2_\phi$
for $A^2(\C;\mu_\phi)$ and the norm of an element $u\in L^2_\phi$
will be denoted by
\begin{equation*}\|u\|_\phi^2=\int_\C\babs{u}^2\e^{-\phi}\dA.\end{equation*}
For a positive integer $n$, we frequently write
\begin{equation}\label{ospaces}L^2_{\phi,n}=\big\{u\in L^2_\phi;\,
u(z)=\calO\lpar\babs{z}^{n-1}\rpar\quad \text{when}\quad
z\to\infty\big\}\end{equation} and
\begin{equation}\label{os2}A^2_{\phi,n}=L^2_{\phi,n}\cap
A^2_\phi=L^2_{\phi,n}\cap\pspace_n.\end{equation} Observe that
$L^2_{\phi,n}$ usually is non-closed in $L^2_\phi$, whereas the
finite-dimensional $A^2_{\phi,n}$ is closed in $L^2_\phi$.

\subsection{The weights considered.} \label{weigh} Let $Q:\C\to \R$ be a given
measurable function, which satisfies a growth condition of the form
\begin{equation}\label{gro}Q(z)\ge \parT\log\babs{z}^2,
\qquad |z|\ge C,
\end{equation}
for some positive numbers $\parT$ and $C$.

For a positive number $m$, we now consider weights $\phi_m$ of the
form $\phi_m=mQ$. By abuse of notation, we will in this context
sometimes refer to $Q$ as the weight. To simplify notation, we set
\begin{equation*}\Hspace_{m,n}=A^2_{mQ,n}=L^2_{mQ}\cap\pspace_n\end{equation*}
and denote by $K_{m,n}$ the reproducing kernel for $\Hspace_{m,n}$.
For a given $z_0$, the corresponding $B_{m,n}^{\langle z_0\rangle}$
is defined accordingly,
cf. \eqref{bm}.

We think of $Q$ as being fixed while the parameters $m$ and $n$
vary, and also fix a number $\tau$ satisfying $0<\tau<\parT$.
To avoid bulky notation, it is customary to reduce the number of
parameters to one, by regarding $n=n(m)$ as a function of $m>0$.
We will
adopt this convention and study the behaviour of the measures
$B^{\langle z_0\rangle}_{m,n}$ when
$m\to\infty$ and $n-m\tau\to 0$.

Adding a real constant to $Q$ means that the inner product in
$\Hspace_{m,n}$ is only chanced by a positive multiplicative
constant, and $K_{m,n}$ gets multiplied by the inverse of that
number. Hence the corresponding Berezin measures are unchanged, and
we can for example w.l.o.g. assume that $Q\ge 1$ on $\C$ when this
is convenient.

\begin{rem} Replacing $Q$ by $\tau^{-1}Q$ one can assume that
$\tau=1$. The most important case (e.g. in random matrix theory)
occurs when $\tau=1$ and $m=n$. However, we shall take on the
slightly more general approach (with three parameters $m$, $n$ and
$\tau$, instead of just $n$) in this note.
\end{rem}

\subsection{A word on notation.} For real $x$, we write $]x[$ for
the largest integer
which is strictly smaller than $x$. We write frequently
$$\d_z=\frac{\d}{\d
z}=\frac1{2}\bigg(\frac{\d}{\d x}-\imag\frac{\d}{\d y}\bigg),\qquad
\dbar_z=\frac{\d}{\d \bar z}=\frac1{2}\bigg(\frac{\d}{\d
x}+\imag\frac{\d}{\d y}\bigg),\qquad z=x+\imag y,$$ and use
$\Delta=\Delta_z$ to denote the normalized Laplacian
$$\Delta_z=\d_z\dbar_z=\frac{1}{4}\bigg(
\frac{\d^2}{\d x^2}+\frac{\d^2}{\d y^2}\bigg).$$ We write
$\D(z_0;r)$ for the open disk $\{z\in\C;\,|z-z_0|<r\}$, and we
simplify the notation to $\D$ when $z_0=0$ and $r=1$. The boundary
of $\D(z_0,r)$ is denoted $\T(z_0,r)$. When $S$ is a subset of $\C$,
we write $S^\circ$ for the interior of $S$, and $\bar{S}$ for its
closure; the support of a continuous function $f$ is denoted by
$\supp f$. The symbol $A\subset B$ means that $A$ is a subset of
$B$, while $A\Subset B$ means that $A$ is a compact subset of $B$.
We write $\dist(A,B)$ for the Euclidean distance between $A$ and
$B$. The symbol \textit{a.e.} is short for "$\diff A$-almost
everywhere'', and \textit{p.m.} is short for "probability measure''.
Finally, we use the shorthand $L^2$ to denote the unweighted space
$L^2_0=L^2(\C;\dA)$.

\subsection{Random normal matrices and the Coulomb gas.} This paper is a slight reworking of a preprint from arxiv.org which was written
in 2007 and early 2008. During the work, it became convenient to collect, in a separate paper, some estimates needed for the paper \cite{AHM}, which was written simultaneously. More precisely, we wanted to make clear a variant of some results on asymptotic expansions of (polynomial) Bergman kernels which were known in a several variables context (\cite{BBS}, \cite{B}), as well as some uniform estimates in the off-diagonal case.
The relevant results for our
applications in \cite{AHM} are primarily theorems \ref{th3} and \ref{flock}. A discussion containing related estimates which hold near the boundary will appear in our later publication \cite{AHM2}.

\section{Main results}
\label{sec-main}

\subsection{Berezin quantization.} Fix a
non-negative
weight $Q$ and two positive numbers $\tau$ and $\parT$,
$\tau<\parT$, such that the growth condition \eqref{gro} is
satisfied, and form the
measures
\begin{equation*}\diff B^{\langle z_0\rangle}_{m,n}=
\berd^{\langle z_0\rangle}_{m,n}\dA\qquad \text{where}\qquad
\berd^{\langle z_0\rangle}_{m,n}(z)=\frac {\babs{K_{m,n}(z,z_0)}^2}
{K_{m,n}(z_0,z_0)}\e^{-mQ(z)}.\end{equation*} We shall refer to the
function $\berd_{m,n}^{\langle z_0\rangle}$ as the \textit{Berezin
kernel} associated with $m$, $n$ and $z_0$.

Let us now assume that $Q\in \calC^2(\C)$.
In the first instance, we ask for which $z_0$ we have convergence
\begin{equation*}B^{\langle z_0\rangle}_{m,n}\to \delta_{z_0}\qquad
\text{as}\quad
m\to\infty \quad \text{and}\quad n-m\tau\to0,\end{equation*} in the
sense of measures. In terms of the \textit{Berezin transform}
\begin{equation*}\berez_{m,n} f(z_0)=\int_\C f(z)\diff B^{\langle
z_0\rangle}_{m,n}(z),
\end{equation*}
we are asking whether
\begin{equation}\label{q2}
\berez_{m,n} f(z_0)\to f(z_0) \qquad \text{for all}\quad
f\in\calC_b(\C)\quad \text{as}\quad
m\to \infty \quad \text{and} \quad n-m\tau\to0.
\end{equation}

Let $\setX$ be the set of points where $Q$ is strictly subharmonic,
$$\setX=\{\Delta Q>0\}.$$
We shall find that there exists a compact set $\setS_\tau$ such that
\eqref{q2} holds for all $z_0$ in the interior of $\setS_\tau\cap
\setX$, while \eqref{q2} fails whenever $z_0\not\in \setS_\tau$.

To define $\setS_\tau$, we first need to introduce some notions from
weighted potential theory, cf. \cite{ST} and \cite{HM}. It is here
advantageous to slightly relax the regularity assumption on $Q$ and
assume that $Q$ is in the class $\calC^{1,1}(\C)$ consisting of all
$\calC^1$-smooth  functions with Lipschitzian gradients. We will
have frequent use of the simple fact that the distributional
Laplacian $\Delta F$ of a function $F\in \calC^{1,1}(\C)$ makes
sense almost everywhere and is of class $L^\infty_{{\rm loc}}(\C)$.

Let $\SH_\tau$ denote the set of all subharmonic functions $f:\C\to
\R$ which satisfy a growth bound of the form
\begin{equation*}f(z)\le\tau\log\babs{z}^2+\calO(1)\quad\text{as}
\quad z\to\infty.\end{equation*} The \textit{equilibrium potential}
corresponding to $Q$ and the parameter $\tau$ is defined by
\begin{equation*}\wh{Q}_\tau(z)=\sup\big\{f(z);\,\,f\in\SH_\tau
\quad\text{and}\quad f\le Q\quad\text{on}\quad
\C\big\}.\end{equation*} Clearly, $\wh{Q}_\tau$ is subharmonic on
$\C$. We now define $\setS_\tau$ as the coincidence set
\begin{equation}\label{coincidence}\setS_\tau=\{Q=\wh{Q}_\tau\}.
\end{equation}
Evidently, $\setS_\tau$ increases with $\tau$, and that $Q$ is
subharmonic on $\setS_\tau^\circ$.

We shall need the following lemma.

\begin{lem}
\label{drspock} The function $\widehat{Q}_\tau$ is of class
$\calC^{1,1}(\C)$, and $\setS_\tau$ is compact. Furthermore,
$\widehat{Q}_\tau$ is harmonic in $\C\setminus \setS_\tau$ and there
is a number $C$ such that $\widehat{Q}_\tau(z)\le
\tau\log_+\babs{z}^2+C$ for all $z\in\C.$
\end{lem}

\begin{proof} This is [\cite{HM}, Prop. 4.2, p. 10]
and [\cite{ST}, Th. I.4.7, p. 54].\end{proof}

\noindent It is easy to construct subharmonic minorants of critical
growth; for $C$ large enough, the function $z\mapsto \tau\log_+\lpar
\babs{z}^2/C\rpar$ is a minorant of $Q$ of class $\SH_\tau$. It
yields that
\begin{equation}\label{qtau}\widehat{Q}_\tau(z)=\tau\log\babs{z}^2+
\calO(1)\quad\text{as}\quad z\to \infty.\end{equation}

\begin{prop}\label{prop1} Let $Q\in \calC^{1,1}(\C)$ and $z_0\in\C$ an arbitrary point.
Then $B_{m,n}^{\langle z_0\rangle}(\Lambda) \to1$ as $m\to\infty$
and $n\le m\tau+ 1$ for every open neighbourhood $\Lambda$ of
$\setS_\tau$.
\end{prop}

\begin{proof} See Sect. \ref{sec3}.\end{proof}

\noindent It follows from Prop. \ref{prop1} that $B_{m,n}^{\langle
z_0\rangle}(\C\setminus\Lambda)\to 0$. Hence if $z_0\not\in
\setS_\tau$, then \eqref{q2} fails in general, and
\begin{equation*}\berez_{m,n}f(z_0)\to 0\qquad \text{as}\quad
m\to\infty\quad \text{and} \quad n\le m\tau+1,\end{equation*} for
all $f\in\calC_b(\C)$ such that $\supp f\cap\setS_\tau=\emptyset$.
The situation is entirely different when the point $z_0$ is in the
interior of $\setS_\tau\cap\setX$.

\begin{thm}\label{th1} Assume that $Q\in\calC^2(\C)$ and let
$z_0\in \setS_\tau^\circ\cap \setX$. Then, for any
$f\in\calC_b(\C)$, and any real number $M\ge 0$, the measures
$B^{\langle z_0\rangle}_{m,n}$ converge to $\delta_{z_0}$ as
$m\to\infty$ and $n\ge m\tau-M$.
\end{thm}

\begin{proof} See Sect. \ref{p2}. See also Rem. \ref{smick}. \end{proof}

\subsection{A more elaborate estimate for the Berezin kernel.}
Th. \ref{th1} suggests that if $z_0$ is a given point in the
interior of $\setS_\tau\cap \setX$, $m$ is large, and $n\ge
m\tau-1$, then the density $\berd_{m,n}^{\langle z_0\rangle}(z)$
should attain its maximum for $z$ close to $z_0$. The following
theorem implies that this is the case, and gives a good control for
$\berd_{m,n}^{\langle z_0\rangle}$ in the critical case, when
$n-m\tau\to 0$.

\begin{thm} \label{th1.5} Assume that $Q\in\calC^2(\C)$. Let $K$ be a compact
subset of $\setS_\tau^\circ\cap\setX$ and fix a non-negative number
$M$. Put
\begin{equation*}d_K=\dist(K,\C\setminus(\setS_\tau\cap\setX))\qquad \text{and}\qquad
a_K=\inf\{\Delta Q(z);\, \dist(z,K)\le d_K/2\}.\end{equation*} Then
there exists positive numbers $C$ and $\epsilon$ such that
\begin{equation*}
\berd^{\langle z_0\rangle}_{m,n}(z)\le C m\e^{-\epsilon
\sqrt{m}\min\{d_K,\babs{z-z_0}\}}\e^{-m(Q(z)-\wh{Q}_\tau(z))},\qquad
z_0\in K,\,\, z\in \C,\,\, m\tau-M\le n\le m\tau+1.
\end{equation*}
Here $C$ depends only on $K$, $M$ and $\tau$, while $\epsilon$ only
depends on $M$, $a_K$ and $\tau$.
\end{thm}

\begin{proof} See Sect. \ref{point}.
\end{proof}

A similar result was proved independently by Berman \cite{B3} after the completion of this note.

\begin{rem} \label{smick} For fixed $M$ and $\tau$, the number
$\epsilon$ can be chosen proportional to $a_K$. A related result
with much more precise information on the dependence of $C$ and
$\epsilon$ on the various parameters in question is given in Th.
\ref{propn1} and the subsequent remark.

We also remark that our proof for Th. \ref{th1.5} is very different
from that for Th. \ref{th1}. Thus Th. \ref{th1.5} gives an
alternative method for obtaining Th. \ref{th1} in the critical case
when $n-m\tau\to 0$.

\end{rem}

\subsection{Gaussian convergence.} Fix a point $z_0\in \setX$. It
will be convenient to introduce the \textit{normalized}
Berezin measure $\wh{B}_{m,n}^{\langle z_0\rangle}$ by
\begin{equation}
\diff\widehat B^{\langle z_0\rangle}_{m,n}=\wh{\berd}_{m,n}^{\langle
z_0\rangle}\dA\quad \text{where}\quad \wh{\berd}_{m,n}^{\langle
z_0\rangle}(z)=\frac 1 {m\Delta Q(z_0)} {\berd^{\langle
z_0\rangle}_{m,n}\bigg(z_0+\frac{z}{\sqrt{m\Delta Q(z_0)}}\bigg)}
.\end{equation} We denote the standard Gaussian p.m. on $\C$ by
\begin{equation*}
\diffP(z)=\e^{-\babs{z}^2}\diff A(z).
\end{equation*}

We have the following CLT, which gives much more precise information
than Th. \ref{th1}. We will settle for stating it for $\calC^\infty$ weights $Q$
which are \textit{real-analytic} in a neighbourhood of $\setS_\tau$, but remark that, using
well-known methods, our proof can be extended to cover the case of,
say, $\calC^\infty$-smooth weights. We will give further details on
possible generalizations later on.

\begin{thm}\label{th2} Assume that $Q$ is real-analytic in a neighbourhood of $\setS_\tau$.
Fix a compact subset $K\Subset \setS_\tau^\circ\cap\setX$, a point
$z_0\in K$, and a number $M\ge 0$. Then we have
\begin{equation}\label{gc1}
\int_\C\babs{\wh{\berd}_{m,n}^{\langle z_0\rangle}(z)
-\e^{-\babs{z}^2}}\dA(z)\to0\qquad\text{as}\quad
m\to\infty\quad \text{and}\quad  n\ge m\tau-M,\end{equation} with
uniform convergence for $z_0\in K$. Equivalently, $\widehat
B^{\langle z_0\rangle}_{m,n}\to P$ as
$m\to\infty$ and $n\ge m\tau-M$.
\end{thm}

\begin{proof} See Sect. \ref{p2}. \end{proof}

\subsection{The expansion formula.} Most of our results in this paper
rely on a suitable approximation for
$K_{m,n}(z,w)$ when $z,w$ are both close to some point
$z_0\in\setX$, and $m$ and $n$ are large. We will here assume that
$Q$ be \textit{real-analytic} in a neighbourhood of $\setS_\tau$.

An adequate and rather far-reaching approximation formula was
recently stated by Berman \cite{B}, p. 9, depending on methods from
\cite{BBS}; we will here only require a special case of his result.
We will need to introduce some definitions.

For subsets $S\subset\C$ we write
$$\adiag(S)=\big\{(z,\bar z)\in\C^2;\,z\in S\big\}.$$
That $Q$ is real-analytic in a neighbourhood of $S$ means that there exists a unique
holomorphic function $\Qext$ defined in a neighbourhood of $\adiag(S)$ in $\C^2$ such that
\begin{equation*}
\Qext(z,\bar{z})=Q(z)\quad\text{for all}\quad  z\in\C.
\end{equation*}
Explicitly, one has
\begin{equation}\label{psi2}\Qext(z+h,\overline{z+k})=\sum_{i,j=0}^\infty
[\d^i\dbar^j Q](z)\frac {h^i \bar{k}^j} {i!j!}\end{equation} when
$h$ and $k$ are sufficiently close to zero. In particular,
$[\d_1^i\d_2^j\Qext](z,\bar{z})=\d^i\dbar^j Q(z)$ for all $z\in\C$.
Here, $\d_1$ and $\d_2$ denote differentiation w.r.t. the first and
second coordinates.

\begin{defn}
\label{def1} Let $\bfun_0$ and $\bfun_1$ denote the functions
\begin{equation*}\bfun_0=\d_1\d_2\Qext,\quad
\bfun_1=\frac12\d_1\d_2 \log\big[\d_1\d_2 \Qext\big]=
\frac{\d_1^2\d_2^2 \Qext \d_1\d_2 \Qext -\d_1^2\d_2 \Qext\d_1\d_2^2
\Qext} {2\big[\d_1\d_2 \Qext\big]^2}.\end{equation*} The function
$\bfun_1$ is well-defined where there is no division by $0$, in
particular in a neighborhood of $\adiag(\setX)$.  Along the
anti-diagonal, we have
\begin{equation*}
\bfun_0(z,\bar{z})=\Delta Q(z)\quad \text{and}\quad
\bfun_1(z,\bar{z})=\frac 12\Delta\log\Delta Q(z),\quad z\in \setX.
\end{equation*}
We note for later use that $\bfun_0$ and $\bfun_1$ are connected via
\begin{equation}\label{funrel}\bfun_1(z,w)=\frac 1 2 \frac \d {\d w}\lpar
\frac 1{\bfun_0(z,w)}\frac \d {\d z}\bfun_0(z,w)\rpar.\end{equation}
We define the \textit{first order approximating Bergman kernel}
$K_m^1(z,w)$ by
\begin{equation*}
K_m^1(z,w)=\lpar m\bfun_0(z,\bar{w})+\bfun_1(z,\bar{w})\rpar \e^{m
\Qext(z,\bar{w})}.\end{equation*}
\end{defn}

We have the following theorem.

\begin{thm}\label{th3} Assume that $Q$ is real-analytic in $\C$.
Let $K$ be a compact subset of $\setS_\tau^\circ\cap\setX$. Fix a
point $z_0\in K$ and a number $M\ge 0$. Then there exists a number
$m_0$ depending only on $M$ and $\tau$ and
a positive number $\eps$ depending only on $K$ and $M$ such that
\begin{equation*}
\babs{K_{m,n}(z,w)-K_m^1(z,w)}\e^{-m(Q(z)+Q(w))/2}\le C m^{-1},
\qquad z_0\in K,\quad z,w\in \D(z_0;\eps),
\end{equation*}
for all $m\ge m_0$ and $n\ge m\tau-M$.
Here $C$ is an absolute constant (depending only on $Q$).
In particular, by restricting
to the diagonal, we get
\begin{equation}\label{ber}\babs{K_{m,n}(z,z)\e^{-mQ(z)}-
\lpar m\Delta Q(z)+\frac 1 2\Delta\log\Delta Q(z)\rpar}\le
\frac{C}{m}, \qquad z\in K,
\end{equation}
when $m\ge m_0$ and $n\ge m\tau-M$.
\end{thm}

\begin{proof} See Sect. \ref{proof}.
See also Remark \ref{bel} below. \end{proof}

\begin{rem} \label{bel}
We want to stress that Th. \ref{th3} has a long history; analogous
expansions are well-known for weighted Bergman kernels, see, for
instance, \cite{F}, \cite{dBMS}, \cite{J}, \cite{E}, \cite{BBS},
\cite{B}, and the references therein. Moreover, as we mentioned
above, Th. \ref{th3} is a slight modification of a more general
result in several complex variables stated by Berman \cite{B}, Th.
3.8, which may be obtained by adapting the methods from \cite{BBS}.
In our proof, we make frequent use of ideas and techniques developed
in \cite{B}, \cite{B2}, \cite{BBS}, and in the book \cite{S}.
\end{rem}

\begin{rem} (Simple properties of $K_m^1$.)

(1) Since $\bfun_0$, $\bfun_1$ and $\Qext$ are real on the
anti-diagonal, a suitable version of the reflection principle
implies that $K_m^1$ is Hermitian, that is, $\overline{
K_m^1(z,w)}=K_m^1(w,z)$.

(2) Using the Taylor series for $\Qext$ at $(z,\bar{z})$ (see
\eqref{psi2}) and ditto for $Q$ and $z$, we get
\begin{equation}\label{bbs}
2\re \Qext(z,\bar{w})-Q(z)-Q(w)=-\Delta
Q(w)\babs{w-z}^2+R(z,w),\end{equation} where
$R(z,w)=\calO\big(\babs{w-z}^3\big)$ for $z$, $w$ in a sufficiently
small neighbourhood of $z_0$, and the $\calO$ is uniform for $z_0$
in a fixed compact subset of $\setX$. In case $\Delta Q(z_0)>0$, it
follows that \begin{equation}\label{uniq}2\re
\Qext(z,\bar{w})-Q(z)-Q(w)\le-\delta_0\babs{w-z}^2,\qquad z,w\in
\D(z_0;2\eps),\end{equation} with $\delta_0=\frac12\Delta Q(z_0)>0$,
provided
that the positive number $\eps$ is chosen small enough. More
generally, if we fix a compact subset $K\Subset \setX$, and put
\begin{equation*}
\delta_0=\frac 1 2\inf_{\zeta\in K}\,\{\Delta
Q(\zeta)\},\end{equation*} we may find an $\eps>0$ such that
\eqref{uniq} holds for all $z_0\in K$. With a perhaps somewhat
smaller $\eps>0$, we may also ensure that
the functions $\bfun_0$ and $\bfun_1$ are bounded and holomorphic in
the set $\{(z,w);\, z,\bar{w}\in\D(z_0;2\eps),\, z_0\in K\}$.
We infer that
\begin{equation}\label{lead}\babs{K_m^1(z,w)}^2\e^{-m(Q(z)+Q(w))}
\le Cm^2\e^{-m\delta_0\babs{z-w}^2},\qquad z,w\in \D(z_0;2\eps),
\quad z_0\in K,
\end{equation}
with a number $C$ depending only on $K$ and $\eps$.
\end{rem}

\subsection{Extensions and possible generalizations.} We discuss
possible extensions of Th. \ref{th3} (and also of Th. \ref{th2}).
Again, these extensions
(at least of Th. \ref{th3}) are essentially implied by a result
[\cite{B}, Th. 3.8], stated by R. Berman.

For a given positive integer $k$, one may define a \textit{$k$-th
order approximating Bergman kernel}
\begin{equation*}K_m^k(z,w)=\lpar m\bfun_0(z,\bar w)+\bfun_1(z,\bar
w)+\ldots+m^{-k+1}\bfun_k(z,\bar w)\rpar\e^{m\Qext(z,\bar
w)},\end{equation*} for $z$ and $w$ in a neighbourhood of
$\adiag(\setX)$ where $\bfun_i$ are certain holomorphic functions
defined in a neighbourhood of $\adiag(\setX)$. It can be shown that
for $z_0\in\setX$ there exists $\eps>0$ such that
\begin{equation*}\babs{K_{m,n}(z,w)-K_m^k(z,w)}^2\e^{-m(Q(z)+Q(w))}\le
Cm^{-2k},\end{equation*} whenever $z,w\in\D(z_0;\eps)$ and $n\ge
m\tau-1$. The coefficients $\bfun_i$ can in principle be determined
from a recursion formula involving partial differential equations of
increasing order, compare with [Berman et al. \cite{BBS}, (2.15),
p.9], where a closely related formula is given. However, the
analysis required for calculating higher order coefficients
$\bfun_i$ for $i\ge 2$ seems to be quite involved, and the first
order approximation seems to be sufficient for many practical
purposes, cf. \cite{AHM}. We therefore prefer a more direct approach
here.

Th. \ref{th3} can also be generalized in another direction -- one
may relax the assumption that $Q$ be real-analytic, and instead
assume that e.g. $Q\in\calC^\infty(\C)$. In this case, the functions
$\bfun_i$ and $\Qext$ will be almost-holomorphic at the
anti-diagonal (cf. e.g. \cite{dBMS} for a relevant discussion of
such functions). The modifications needed for proving Th. \ref{th3}
in this more general case are based on standard arguments; they are
essentially as outlined in [\cite{BBS}, Subsect. 2.6 and p. 15]. We
leave the details to the interested reader. However, the reader
should note that, using this generalized version of Th. \ref{th3},
one may easily extend our proof of Th. \ref{th2}, to the case of
$\calC^\infty$-smooth weights.

\subsection{The Bargmann--Fock case and harmonic measure.} When
$z_0\in\C\setminus(\setS_\tau^\circ\cap \setX)$, Prop. \ref{prop1}
yields that the measures $B^{\langle z_0\rangle}_{m,n}$ tend to
concentrate on $\setS_\tau$ as
$m\to\infty$ and $n-m\tau\to 0$.
However our general results provide no further information regarding
the asymptotic distribution.

We now specialize to the Bargmann--Fock weight $Q(z)=\babs{z}^{2}$.
We then obviously have $\setX=\C$. It will be convenient to
introduce the functions (truncated exponentials)
\begin{equation*}E_k(z)=\sum_{j=0}^k\frac {z^j} {j!}.\end{equation*}
It is then easy to check that $K_{m,n}(z,w)=mE_{n-1}(mz\bar{w}),$
$\setS_\tau=\overline{\D}(0;\sqrt{\tau}),$ and
$\widehat{Q}_\tau(z)=\tau+\tau\log\lpar\babs{z}^2/\tau\rpar$ when
$\babs{z}^2\ge\tau.$ We infer that
\begin{equation}\label{bfock}\diff B^{\langle z_0\rangle}_{m,n}(z)=
m\frac{\babs{E_{n-1}(mz\bar{z}_0)}^2} {E_{n-1}(m\babs{z_0}^2)}
\e^{-m\babs{z}^2}\diff A(z).
\end{equation}
Let $\C^*=\C\cup\{\infty\}$ denote the extended plane.

\begin{thm}\label{th5} Let $Q(z)=\babs{z}^2$ and $z_0\in\C\setminus \setS_\tau$. Then, as
$m\to\infty$ and $n/m\to\tau$, the measures $B^{\langle
z_0\rangle}_{m,n}$ converge to harmonic measure at $z_0$ with
respect to $\C^*\setminus\setS_\tau$.
\end{thm}

\begin{proof} See Sect. \ref{new}. \end{proof}

\begin{rem} The above theorem is a special case of Theorem 8.7 in \cite{AHM}.
\end{rem}

\subsection{Weighted $L^2$ estimates for the $\dbar$-equation with a
growth constraint.} Our analysis depends in an essential way on
having a good control for the norm-minimal solution $u_*$ in the
space $L^2_{mQ,n}$ to the $\dbar$-equation $\dbar u=f$, where $f$ is
a given suitable test-function such that $\supp f\subset
\setS_\tau\cap\setX$. We would not have a problem if $u_*$ were just
required to be of class $L^2_{mQ}$ (and not $L^2_{mQ,n}$), for then
an elementary version of the well-known $L^2$ estimates of
H\"ormander apply, see e.g. \cite{H} or \cite{H2}. The requirement
that $u_*$ be of restricted growth gives rise to some difficulties,
but we shall see in Sect. \ref{l2estimates} (notably Th. \ref{strb})
that the classical estimates can be adapted to our situation.
(Estimates very similar to ours were given independently by Berman in \cite{B3} after this note was completed.)

\section{Preparatory lemmas} \label{sec3}

\subsection{Preliminaries.} In this section, we prove Prop. \ref{prop1}. At the
same time we will get the opportunity to display some elementary
inequalities which will be used in later sections.
For the
results obtained here, it is enough to suppose that $Q$ is
$\calC^{1,1}$-smooth (and satisfies the growth assumption
\eqref{gro}). Thus $\Delta Q$ makes sense almost everywhere and is
of class $L^\infty_{\rm loc}(\C)$.

\subsection{Subharmonicity estimates.} We start by giving two
simple lemmas which are based on the sub-mean property of
subharmonic functions.

\begin{lem} \label{berndt} Let $\phi$ be a weight of class
$\calC^{1,1}(\C)$ and put
$A=\esssup\{\Delta\phi(\zeta);\,\zeta\in\D\}.$ Then if $u$ is
bounded and holomorphic in $\D$, we have that
$$|u(0)|^2\e^{-\phi(0)}\le\int_\D|u(\zeta)|^2\e^{-\phi(\zeta)}\e^{A|\zeta|^2}\dA(\zeta)
\le\e^A\int_\D|u(\zeta)|^2\e^{-\phi(\zeta)}\dA(\zeta).$$
\end{lem}

\begin{proof} Consider the function
$F(\zeta)=|u(\zeta)|^2\e^{-\phi(\zeta)+A|\zeta|^2},$ which satisfies
$$\Delta \log F=\Delta\log|u|^2-\Delta\phi+A\ge0\qquad \text{a.e.\, on}\quad \D,$$
making $F$ logarithmically subharmonic. But then $F$ is subharmonic
itself, and so
$$F(0)\le\int_\D F(z)\dA(z).$$
The assertion of the lemma is immediate from this.
\end{proof}

\begin{lem} \label{lemm2} Let $Q\in \calC^{1,1}(\C)$ and fix $\delta>0$ and $z\in\C$ and put
$A=A_{z,\delta}=\esssup\{\Delta Q(\zeta);\, \zeta\in\D(z;\delta)\}.$
Also, let $u$ be holomorphic and bounded in $\D(z;\delta/\sqrt{m})$.
Then
\begin{equation*}\babs{u(z)}^2e^{-mQ(z)}\le
m\delta^{-2}{\e^{A\delta^2}}
\int_{\D(z;\delta/\sqrt{m})}\babs{u(\zeta)}^2e^{-mQ(\zeta)}\dA(\zeta).
\end{equation*}
\end{lem}

\begin{proof} The assertion follows if we make the
change of variables $\zeta=z+\delta\xi/\sqrt{m}$ where $\zeta\in
\D(z;\delta/\sqrt{m})$ and $\xi\in\D$, and apply Lemma \ref{berndt}
with the weight $\phi(\xi)=mQ(\zeta)$.
\end{proof}

\noindent We note the following consequence of the subharmonicity
estimates. We will frequently need it in later sections.

\begin{lem} \label{lemm3} Let $K$ be a compact subset of $\C$ and
$\delta$ a given positive number. We put
\begin{equation*}K_\delta=\{z\in\C;\, \dist(z,K)\le \delta\}\qquad \text{and}
\qquad A=\esssup\{\Delta Q(z);\, z\in K_\delta\}.\end{equation*}
Then, for all $m,n\ge 1$,
\begin{equation*}\babs{K_{m,n}(z,w)}^{2}\e^{-mQ(z)}\le m\delta^{-2}\e^{\delta^2
A}\int_{\D(z;\delta/\sqrt{m})}\babs{K_{m,n}(\zeta,w)}^2\e^{-mQ(\zeta)}\dA(\zeta),\quad
z\in K,\quad w\in \C.\end{equation*}
\end{lem}

\begin{proof} Apply Lemma \ref{lemm2} to $u(\zeta)=K_{m,n}(\zeta,w)$.
\end{proof}

\subsection{A weak maximum principle for weighted polynomials.}
Maximum principles for weighted polynomials have a long history, see
e.g. \cite{ST}, Chap. III. The following simple lemma will suffice
for our present purposes; it is a consequence of \cite{ST}, Th.
III.2.1. (Recall that $\setS_\tau$ denotes the coincidence set
$\{Q=\wh{Q}_\tau\}$, see \eqref{coincidence}.)

\begin{lem} \label{wmax} Suppose that a polynomial $u$ of degree at most
$n-1$ satisfies $\babs{u(z)}^2\e^{-mQ(z)}\le 1$ on
$\setS_{\tau(m,n)}$, where $\tau(m,n)=(n-1)/m<\parT$. Then
$\babs{u(z)}^2\e^{-m\wh{Q}_{\tau(m,n)}(z)}\le 1$ on $\C$.
\end{lem}

\begin{proof} Put $q(z)=m^{-1}\log\babs{u(z)}^2$. The assumptions
on $u$ imply that $q\in\SH_{\tau(m,n)}$ and that $q\le Q$ on
$\setS_{\tau(m,n)}$. Hence $q\le\wh{Q}_{\tau(m,n)}$ on $\C$, as
desired.
\end{proof}

\begin{lem} \label{lemm4} Let
$u\in\Hspace_{m,n}$, and suppose that $n\le m\tau+1$. Then
\begin{equation*}\babs{u(z)}^2\le m\e^A\|u\|_{mQ}^2\e^{m\wh{Q}_\tau(z)},
\end{equation*}
where $A$ denotes the essential supremum of $\Delta Q$ over the set
$\{z\in\C;\, \dist(z,\setS_\tau)\le 1\}$.
\end{lem}

\begin{proof} The assertion that $n\le m\tau+1$ is equivalent to that
$\tau(m,n)\le \tau$ where $\tau(m,n)=(n-1)/m$. Thus
$\setS_{\tau(m,n)}\subset\setS_\tau$ and $\wh{Q}_{\tau(m,n)}\le
\wh{Q}_\tau$. An application of Lemma \ref{lemm2} with $\delta=1$
now gives
\begin{equation*}\babs{u(z)}^2\e^{-mQ(z)}
\le m\e^A\int_{\D(z;1)}\babs{u(\zeta)}^2
\e^{-mQ(\zeta)}\dA(\zeta)\le
m\e^A\int_\C\babs{u(\zeta)}^2\e^{-mQ(\zeta)} \dA(\zeta),\quad
z\in\setS_{\tau}.
\end{equation*}
As a consequence, the same estimate holds on $\setS_{\tau(m,n)}$.
We can thus apply Lemma \ref{wmax}. It yields that
\begin{equation*}\babs{u(z)}^2\le m\e^A\|u\|_{mQ}^2\e^{m\wh{Q}_{\tau(m,n)}(z)},
\quad z\in\C.\end{equation*}
The desired assertion follows, since $\wh{Q}_{\tau(m,n)}\le\wh{Q}_\tau$.
\end{proof}

\subsection{The proof of Proposition \ref{prop1}} Fix two points $z$ and
$z_0$ in $\C$. We apply Lemma \ref{lemm4} to the polynomial
\begin{equation*}u(\zeta)=\frac {K_{m,n}(\zeta,z_0)}{\sqrt{K_{m,n}(z_0,z_0)}},
\end{equation*}
which is of class $\Hspace_{m,n}$ and satisfies $\|u\|_{mQ}=1$. It
yields that $\babs{u(z)}^2\le m\e^A\e^{\wh{Q}_\tau(z)}$, or
\begin{equation}\label{cerd}
\berd_{m,n}^{\langle w\rangle}(z)=\babs{u(z)}^2\e^{-mQ(z)} \le
m\e^{A}\e^{m(\wh{Q}_\tau(z)-Q(z))},\quad z\in\C,\, n\le
m\tau+1.\end{equation}

Now let $\Lambda$ be an open neighborhood of $\setS_\tau$. Since
$Q>\widehat{Q}_\tau$ on $\C\setminus \setS_\tau$, the continuity of
the functions involved coupled with the growth conditions
\eqref{qtau} and \eqref{gro} yield that $Q-\widehat{Q}_\tau$ is
bounded below by a positive number on $\C\setminus\Lambda$. It
follows that \begin{equation*}B_{m,n}^{\langle
z_0\rangle}(\C\setminus\Lambda)=
\int_{\C\setminus\Lambda}\berd^{\langle z_0\rangle}_{m,n}\diff A\to
0\end{equation*} as
$m\to\infty$ and $n\le m\tau+1$.
Since $B_{m,n}^{\langle z_0\rangle}$ is a p.m. on $\C$, the
assertion of Prop. \ref{prop1} follows. \hfill$\qed$

\subsection{An estimate for the one-point function.} Our proof of
Prop. \ref{prop1} implies a useful estimate for the one-point
function $z\mapsto K_{m,n}(z,z)\e^{-mQ(z)}$. The following result is
implicit in \cite{B}.

\begin{prop}\label{lemma1} Suppose that $n\le m\tau+1$.
Then there exists a number $C$ depending only on $\tau$ such that
\begin{equation}\label{berg1}K_{m,n}(z,z)\e^{-mQ(z)}\le Cm
\e^{-m(Q(z)-\wh{Q}_\tau(z))},\quad z\in\C,\end{equation}
\begin{equation}\label{berg2}\babs{K_{m,n}(z,w)}^2\e^{-m(Q(z)+Q(w))}\le
Cm^2\e^{-m(Q(z)-\wh{Q}_\tau(z))}\e^{-m(Q(w)-\wh{Q}_\tau(w))},\quad
z,w\in\C.
\end{equation}
In particular, $\babs{K_{m,n}(z,w)}^2\e^{-m(Q(z)+Q(w))}\le Cm^2$ for
all $z$ and $w$.
\end{prop}

\begin{proof} The
function $\berd_{m,n}^{\langle z_0\rangle}(z)$ in the diagonal case
$z=z_0$ reduces to
\begin{equation*}\berd_{m,n}^{\langle z\rangle}(z)=K_{m,n}(z,z)\e^{-mQ(z)}.
\end{equation*}
Thus the estimate \eqref{cerd} implies
\begin{equation*}K_{m,n}(z,z)\e^{-mQ(z)}\le m\e^A\e^{m(\wh{Q}_\tau(z)-Q(z))},
\end{equation*}
which proves \eqref{berg1}. In order to get \eqref{berg2} it now
suffices to use the Cauchy--Schwarz inequality
$\babs{K_{m,n}(z,w)}^2\le K_{m,n}(z,z)K_{m,n}(w,w)$ and apply
\eqref{berg1} to the two factors in the right hand side.
\end{proof}

\subsection{Cut-off functions.} \label{cutoff} We will in the following
frequently use cut-off functions with various properties. For
convenience of the reader who may not be familiar with these
details, we collect the necessary facts here.

Given any numbers $\delta>0$, $r>0$ and $C>1$, there exists $\chi\in
\coity(\C)$ such that $\chi=1$ on $\D(0;\delta)$, $\chi=0$ outside
$\D(0;\delta(1+r))$, $\chi\le 1$ and $|\dbar\chi|^2\le (C/r^2)\chi$
on $\C$. \footnote{Such a $\chi$ can be obtained by standard
regularization applied to the Lipschitzian function
which equals $(1-(\delta^{-1}\babs{z}-1)/r)^2$ for $1\le\babs{z}\le
r$, and is otherwise locally constant on $\C$.} Moreover, with such
a choice for $\chi$, it follows that $|\dbar\chi|^2\le
(C/r^2)\delta^{-2}\chi$ on $\C$. It is then easy to check that
$\|\dbar\chi\|_{L^2}^2\le C(1+2/r)$.

Later on, we will often use the values $\delta=3\eps/2$ and
$\delta(1+r)=2\eps$, where $\eps>0$ is given. We may then arrange
that
$\|\dbar\chi\|_{L^2}\le 3$.

\section{Weighted estimates for the $\dbar$-equation with a
growth constraint}\label{l2estimates}

\subsection{General introduction; the Bergman projection and the
$\dbar$-equation.} \label{init} Let $\phi$ be a weight on $\C$ (cf.
Subsect. \ref{weigh}). We assume throughout that $\phi$ is of class
$\calC^{1,1}(\C)$ (so that $\Delta\phi\in L^\infty_{{\rm
loc}}(\C)$), and that $\int_\C\e^{-\phi}\diff A<\infty$ (so that
$L^2_\phi$ contains the constant functions). Also fix a positive
integer $n$ and recall the definition of the "truncated'' spaces
$L^2_{\phi,n}$ and $A^2_{\phi,n}$ (see \eqref{ospaces} and
\eqref{os2}).

 Let $K_{\phi,n}$ denote the
reproducing kernel for $A^2_{\phi,n}$, and let
$\proj_{\phi,n}:L^2_\phi\to A^2_{\phi,n}$ be the orthogonal
projection,
\begin{equation}\label{ooproj}\proj_{\phi,n}u(w)=\int_\C
u(z)K_{\phi,n}(w,z)\e^{-\phi(z)}\dA(z),\quad u\in
L^2_{\phi}.\end{equation} We will in later sections frequently need
to estimate $\proj_{\phi,n}u(w)$, especially when $u$ is holomorphic
in a neighbourhood of $w$.

Now note that for $f$ in the class $\calC_0(\C)$ (continuous
functions with compact support), the Cauchy transform $u=\cauchy f$
given by
\begin{equation*}\cauchy f(w)=\int_\C\frac {f(z)} {w-z}
\dA(z),\end{equation*} satisfies the $\dbar$-equation
\begin{equation}\label{dstreck}\dbar u=f.\end{equation}
Moreover, $u=\cauchy f$ is bounded, and is therefore of class
$L^2_{\phi,n}$ for any $n\ge 1$.
Thus the function
\begin{equation}\label{star}u_{*}(w)=u(w)-\proj_{\phi,n}u(w),\end{equation}
solves \eqref{dstreck} and is of class $L^2_{\phi,n}$. It is easy to
verify that $u_{*}$ defined by \eqref{star} is the unique
norm-minimal solution to \eqref{dstreck} in $L^2_{\phi,n}$ whenever
$u\in L^2_{\phi,n}$ is a solution to \eqref{dstreck}. Hence the
study of the orthogonal projection $\proj_{\phi,n}u$ is equivalent
to the study of the \textit{the $L^2_{\phi,n}$-minimal} solution
$u_{*}$ to \eqref{dstreck}.

It is useful to observe that $u_{*}$ is characterized amongst the
solutions of class $L^2_{\phi,n}$ to \eqref{dstreck} by the
condition $u_{*}\bot A^2_{\phi,n}$, or
\begin{equation}\label{orthog}
\int_\C u_{*}(z)\overline{h(z)}\e^{-\phi(z)}\dA(z)=0,\quad\text{for
all}\quad h\in A^2_{\phi,n}.\end{equation}

Our
principal result in this section, Th. \ref{strb}, states that a
variant of the elementary one-dimensional form of the $L^2$
estimates of H\"ormander are valid for $u_{*}$. The important
results for the further developments in this paper are however Th.
\ref{boh} and Cor. \ref{bh}. The reader may consider to glance at
those results and skip to the next section, at a first read.

\subsection{$L^2$ estimates.} The
$L^2$ estimates of H\"{o}rmander in the most elementary,
one-dimensional form applies only to weights $\phi$ which are
strictly subharmonic in the entire plane $\C$. The result states
that $u_0$, the $L^2_\phi$-minimal solution to \eqref{dstreck}
(where $f\in\calC_0(\C)$) satisfies
\begin{equation}\label{hoe0}\int_\C\babs{u_0}^2\e^{-\phi}\dA\le\int_\C
\babs{f}^2\frac {\e^{-\phi}} {\Delta\phi}\dA,\end{equation} provided
that $\phi$ is $\calC^2$-smooth on $\C$. See \cite{H}, eq. (4.2.6),
p. 250 (this is essentially just Green's formula).

It is important to observe that the estimate \eqref{hoe0} remains
valid for strictly subharmonic weights $\phi$ in the larger class
$\calC^{1,1}(\C)$ (that $\phi$ is strictly subharmonic then means
that $\Delta\phi>0$ a.e. on $\C$). The proof in \cite{H}, Subsect.
4.2 goes through without changes in this more general case.

\subsection{Weighted $L^2_{\phi,n}$ estimates.} We have the following theorem, where
we consider two weights $\phi$ and $\wh{\phi}$ with various
properties. The practically minded reader may, with little loss of
generality, think of $\phi=mQ$,
$\wh{\phi}(z)=m\wh{Q}_\tau(z)+\eps\log(1+\babs{z}^2)$, and
$\setSp=\setS_\tau$. This will essentially be the case in all our
later applications.

\begin{thm}\label{strb} Let $\setSp$ be a compact subset of $\C$,
$\phi$, $\wh\phi$ and $\Tfun$ three real-valued functions
of class $\calC^{1,1}(\C)$, and $n$ a positive integer. We assume
that:
\begin{enumerate}
\item[(1)] $L^2_{\wh{\phi}}$ contains the constants and there are
non-negative numbers $\alpha$ and $\beta$ such that
\begin{equation}\label{put}\wh{\phi}\le
\phi+\alpha\quad\text{on}\quad \C\qquad\text{and}\qquad\phi\le
\beta+\wh{\phi}\quad \text{on}\quad \setSp,\end{equation}

\item[(2)] $A^2_{\wh{\phi}}\subset A^2_{\phi,n}$,

\item[(3)] The function $\wh{\phi}+\Tfun$ is strictly subharmonic on
$\C$,

\item[(4)] $\Tfun$ is locally constant on $\C\setminus \setSp$,

\item[(5)] There exists a number $\kappa$, $0<\kappa<1$, such that
\begin{equation}\label{kappa}
\frac{|\dbar \Tfun|^2}{\Delta\wh{\phi}+\Delta\Tfun}\le
\frac{\kappa^2}{\e^{\alpha+\beta}}\quad \text{a.e.   on}\quad
\setSp.\end{equation}

\end{enumerate}
Let $f\in \coity(\C)$ be such that
\begin{equation*}\supp f\subset \setSp.\end{equation*} Then
$u_*$, the $L^2_{\phi,n}$-minimal solution to $\dbar u=f$ satisfies
\begin{equation}\label{zapp}\int_\C\babs{u_*}^2\e^{\Tfun-\phi}\dA\le \frac
{\e^{\alpha+\beta}} {(1-\kappa)^2}\int_\C\babs{f}^2\frac
{\e^{\Tfun-\phi}} {\Delta\wh{\phi}+\Delta\Tfun} \dA.\end{equation}
\end{thm}

\noindent Before we give the proof, we note a simple lemma, which
will be put to repeated use.

\begin{lem}\label{convo} Suppose that $f\in\calC_0(\C)$. Then
\eqref{dstreck} has a solution $u$ in $L^2_{\wh{\phi}}$. Moreover
$v_*$, the $L^2_{\wh{\phi}}$-minimal solution to \eqref{dstreck} is
of class $L^2_{\phi,n}$.
\end{lem}

\begin{proof}
The assumptions imply that the
Cauchy transform $\cauchy f$ is in $L^2_{\wh\phi}$. Thus the
$L^2_{\wh\phi}$-minimal solution $v_*$ to \eqref{dstreck} exists,
necessarily of the form $v_*=\cauchy f+g$ with some $g\in
A^2_{\wh\phi}$. In view of (3), we know that $g\in A^2_{\phi,n}$.
The assertion is now immediate.
\end{proof}

\begin{proof}[Proof of Theorem \ref{strb}]
Assume that the right hand side in \eqref{zapp} is finite. In view
of \eqref{orthog}, the condition that $u_{*}$ is
$L^2_{\phi,n}$-minimal may be expressed as
\begin{equation*}\int_\C u_{*}\e^\Tfun \bar{h} \e^{-(\phi+\Tfun)}\dA=0
\quad\text{for all}\quad h\in A^2_{\phi,n}.
\end{equation*}
The latter relation means that the function $w_{0}=u_{*}\e^\Tfun$
minimizes the norm in $L^2_{\phi+\Tfun}$ over elements of the
(non-closed) subspace
\begin{equation*}\e^{\Tfun}\cdot L^2_{\phi,n}=\big\{w;\quad w=h\e^\Tfun,\quad\text{where}\quad h\in
L^2_{\phi,n}\big\}\subset L^2_{\phi+\Tfun}
\end{equation*}
which solve the (modified) $\dbar$-equation
\begin{equation}\label{mdbar}\dbar w=\dbar\lpar u_{*}\e^\Tfun\rpar=
f\e^\Tfun+u_{*}\dbar\lpar \e^\Tfun\rpar.
\end{equation}
Now, since $\Tfun$ is bounded on $\C$ and locally constant outside
$\setSp$, we have
\begin{equation*}L^2_{\phi+\Tfun,n}=\e^\Tfun\cdot L^2_{\phi,n}=
L^2_{\phi,n} \qquad (\text{as sets}),\end{equation*} and we conclude
that $w_{0}$ is the norm-minimal solution to \eqref{mdbar} in
$L^2_{\phi+\Tfun,n}$.

Let $v_*$ denote the $L^2_{\wh{\phi}}$-minimal solution to $\dbar
u=f$ (see Lemma \ref{convo}). We form the continuous function
$g=\dbar\lpar(u_{*}-v_*)\e^\Tfun\rpar=(u_{*}-v_*)\dbar\lpar
\e^\Tfun\rpar,$ whose support is contained in $\setSp$, and consider
the $\dbar$-equation
\begin{equation}\label{go}\dbar\xi=g=(u_{*}-v_*)\dbar\lpar \e^\Tfun\rpar.
\end{equation}
The assertion of Lemma \ref{convo} remains valid if $\wh\phi$ is
replaced by $\wh{\phi}+\Tfun$; it follows that \eqref{go} has a
solution $\xi\in L^2_{\wh{\phi}+\Tfun}$, and moreover, if $\xi_{0}$
denotes the norm-minimal solution to \eqref{go} in
$L^2_{\wh{\phi}+\Tfun}$, we have $\xi_{0} \in L^2_{\phi+\Tfun,n}$.

We now continue by forming the function
\begin{equation*}w_1=v_*\e^{\Tfun}+\xi_{0}.\end{equation*}
It is clear that $w_1\in L^2_{\phi+\Tfun,n}$, and a calculation
yields that
\begin{equation}\label{nu2}\dbar w_1=
f \e^{\Tfun}+v_*\dbar\lpar \e^{\Tfun}\rpar +(u_{*}-v_*)\dbar\lpar
\e^{\Tfun}\rpar =\dbar\lpar u_{*}\e^{\Tfun}\rpar=\dbar w_{0}.
\end{equation}
Since $w_{0}$
is norm-minimal in $L^2_{\phi+\Tfun,n}$ over functions $w$ such that
$\dbar w=\dbar w_{0}$, we must have
\begin{equation}
\label{ojoj}\int_\C\babs{w_{0}}^2 e^{-(\phi+\Tfun)}\dA\le
\int_\C\babs{w_1}^2\e^{-(\phi+\Tfun)}\dA\le \e^{\alpha}\int_\C
\babs{w_1}^2 \e^{-(\wh{\phi}+\Tfun)}\dA,
\end{equation}
where we have used the condition $\wh{\phi}\le\phi+\alpha$ to deduce
the second inequality. Moreover, since $\xi_0$
is norm-minimal in $L^2_{\wh{\phi}+\Tfun}$ amongst
solutions to
\eqref{go}, we have
\begin{equation*}
\int_\C w_1\bar{h}\e^{-(\wh{\phi}+\Tfun)}\dA=\int_\C
v_*\bar{h}\e^{-\wh{\phi}}\dA+ \int_\C \xi_0
\bar{h}\e^{-(\wh{\phi}+\Tfun)}\dA=0\quad
\end{equation*}
for all $h\in A^2_{\wh{\phi}+\Tfun}$, so the function $w_1$ is in
fact
the norm-minimal solution to a $\dbar$-equation in
$L^2_{\wh{\phi}+\Tfun}$. The $\dbar$-equation satisfied by $w_1$ is
(see \eqref{nu2})
\begin{equation*}\dbar w_1=f\e^\Tfun+u_{*}\dbar\lpar \e^\Tfun\rpar =
f \e^\Tfun +u_{*}\e^\Tfun\dbar \Tfun.
\end{equation*}
By the estimate \eqref{hoe0} applied to the weight
$\wh\phi+\Tfun$, we obtain that
\begin{equation}\label{inter}\int_\C\babs{w_1}^2
\e^{-(\wh{\phi}+\Tfun)}\dA\le \int_{\C} \babs{f\e^\Tfun
+u_{*}\e^\Tfun\dbar \Tfun}^2 \frac {\e^{-(\wh{\phi}+\Tfun)}}
{\Delta\lpar \wh{\phi}+\Tfun\rpar}\dA=\int_{\C} \babs{f
+u_{*}\dbar\Tfun}^2 \frac {\e^{\Tfun-\wh{\phi}}}
{\Delta\wh{\phi}+\Delta\Tfun}\dA.
\end{equation}
Since $f$ and $\dbar\Tfun$ are supported in $\setSp$, and since
$\e^{-\wh{\phi}}\le \e^{\beta}\e^{-\phi}$ there (see \eqref{put}),
it is seen that the right hand side in \eqref{inter} can be
estimated by
\begin{equation}\label{inter2}\e^{\beta}
\int_\C\babs{f+u_*\dbar\Tfun}^2\frac {\e^{\Tfun-\phi}}
{\Delta\wh{\phi}+\Delta\Tfun} \dA.\end{equation} For $t>0$ we now
use the inequality $\babs{a+b}^2\le (1+t)\babs{a}^2+
(1+t^{-1})\babs{b}^2$
and the condition \eqref{kappa} to conclude that the
integral in \eqref{inter2} is dominated by
\begin{equation}\label{inter3}\begin{split}(1+t)&\int_\C
\babs{f}^2\frac {\e^{\Tfun-\phi}} {\Delta\wh{\phi}+\Delta\Tfun}
\dA+(1+t^{-1})\int_\C\babs{u_*}^2\frac {|\dbar\Tfun|^2}
{\Delta\wh{\phi}+\Delta\Tfun}
\e^{\Tfun-\phi}\dA\le\\
&\le (1+t)\int_\C\babs{f}^2\frac {\e^{\Tfun-\phi}}
{\Delta\wh{\phi}+\Delta\Tfun}\dA+(1+t^{-1})\frac {\kappa^2}
{\e^{\alpha+\beta}}
\int_\C\babs{u_*}^2\e^{\Tfun-\phi}\dA.\\
\end{split}
\end{equation}
Tracing back through \eqref{ojoj}, \eqref{inter}, \eqref{inter2} and
\eqref{inter3}, we get
\begin{equation*}\int_\C\babs{u_*}^2\e^{\Tfun-\phi}\dA\le
\e^{\alpha+\beta}(1+t)\int_\C\babs{f}^2\frac {\e^{\Tfun-\phi}}
{\Delta\wh{\phi}+\Delta\Tfun}\dA+
(1+t^{-1})\kappa^2\int_\C\babs{u_*}^2\e^{\Tfun-\phi}\dA,\end{equation*}
which we write as
\begin{equation*}
\big(1-(1+t^{-1})\kappa^2\big)\int_\C|u_{*}|^2\e^{\Tfun-\phi}\dA
 \le\e^{\alpha+\beta}(1+t)\int_{\C}\babs{f}^2\frac{\e^{\Tfun-\phi}}
{\Delta\wh{\phi}+\Delta\Tfun}\dA.
\end{equation*}
The desired estimate \eqref{zapp} now follows if we pick
$t=\kappa/(1-\kappa).$
\end{proof}

\subsection{Implementation scheme.} We now fix
$Q\in \calC^{1,1}(\C)$ satisfying the growth assumption \eqref{gro}
with a fixed $\parT>0$. Adding a constant to $Q$ does not change the
problem and so we may assume that $Q\ge 1$ on $\C$. Let us put
\begin{equation}\label{sigmat}q_\tau=\sup_{z\in\setS_\tau}\{Q(z)\}.\end{equation}
We next fix a positive number $\tau<\parT$ and two positive numbers $\apar$
and $\bpar$ such that
\begin{equation}\label{sss}
\bpar\log(1+\babs{z}^2)\le\apar \wh{Q}_\tau(z),\qquad
z\in\C.\end{equation}
This is possible since $\wh{Q}_\tau\ge 1$ and
$\wh{Q}_\tau(z)=\tau\log\babs{z}^2+\calO(1)$ when $z\to\infty$ (see
\eqref{qtau}). In particular, it yields that $\bpar\le \apar\tau$.
We now define
\begin{equation}\label{pha}\phi_m=mQ\quad
\text{and}\quad
\wh{\phi}_{m,\apar,\bpar}(z)=(m-\apar)\widehat{Q}_\tau(z)+\bpar\log(1+\babs{z}^2).\end{equation}
Note that $\wh\phi_{m,\apar,\bpar}$ is strictly subharmonic on $\C$
whenever $m\ge \apar$ with
\begin{equation}\label{bye}\Delta \wh{\phi}_{m,\apar,\bpar}(z)=(m-\apar)\Delta \wh{Q}_\tau(z)+
\bpar\lpar 1+\babs{z}^2\rpar^{-2} \ge \bpar\lpar
1+\babs{z}^2\rpar^{-2},\end{equation} and that \eqref{sss} implies
(since $\wh{Q}_\tau\le Q$)
\begin{equation*}\wh{\phi}_{m,\apar,\bpar}\le \phi_m\qquad
\text{on}\quad \C.\end{equation*} Furthermore, \eqref{qtau} implies
that
\begin{equation}\label{qq}\wh{\phi}_{m,\apar,\bpar}=\lpar(m-\apar)\tau+\bpar\rpar
\log\babs{z}^2+\calO(1), \qquad \text{as}\quad
z\to\infty.\end{equation} It yields that
\begin{equation*}\phi_m(z)-\wh{\phi}_{m,\apar,\bpar}(z)\ge
m(\parT-\tau)\log\babs{z}^2+\calO(1),\qquad\text{as}\quad
z\to\infty.\end{equation*} Note also that
\begin{equation*}A^2_{\phi_m,n}=A^2_{mQ}\cap\calP_n=\Hspace_{m,n}.
\end{equation*}
We now check the conditions (1) and (2) of Th. \ref{strb}. (Recall
that $]x[$ denotes the largest integer which is strictly smaller
than $x$.)

\begin{lem} \label{putty}
Condition {\rm (1)} holds for $\phi=\phi_m$,
$\wh\phi=\wh\phi_{m,\apar,\bpar}$, $\setSp=\setS_\tau$, $\alpha=0$
and
$\beta=\apar q_\tau$, provided that $(m-\apar)\tau+\bpar>1$.
Condition {\rm (2)}, i.e. $A^2_{\wh{\phi}_{m,\apar,\bpar}}\subset
\Hspace_{m,n}$, holds if $n\ge ](m-\apar)\tau+\bpar[$. Thus
conditions {\rm (1)} and {\rm (2)} hold whenever
\begin{equation}\label{alag}n\ge](m-\apar)\tau+\bpar[>0.\end{equation}
\end{lem}

\begin{proof}
We have already shown that $\wh{\phi}_{m,\apar,\bpar}\le \phi_m$ on
$\C$. Moreover $\phi_m\le\beta+\wh{\phi}_{m,\apar,\bpar}$ on
$\setS_\tau$, since $Q=\wh{Q}_\tau$ there. The assertion that
$\wh{\phi}_{m,\apar,\bpar}\le \phi_m$ implies that
$A^2_{\wh{\phi}_{m,\apar,\bpar}}\subset A^2_{\phi_m}$, and it
remains only to prove that $A^2_{\wh{\phi}_{m,\apar,\bpar}}\subset
\calP_n$. But this follows from \eqref{qq}, \eqref{alag}, and the
fact that
$\int_{\C}(1+|z|^{2})^{-r}\dA(z)<\infty$ if and only if $r>1$.
\end{proof}

We now apply Th. \ref{strb}.
It will be convenient to define
\begin{equation}\label{ctdef}c_\tau=\inf_{z\in\setS_\tau}
\big\{\lpar 1+\babs{z}^2\rpar^{-2}\big\}.\end{equation}

\begin{thm} \label{boh} Let $Q\in\calC^{1,1}(\C)$ and $Q\ge 1$ on $\C$. Fix two positive
numbers $\apar$ and $\bpar$ such that relation \eqref{sss} is
satisfied, and define the positive numbers $q_\tau$ and $c_\tau$ by
\eqref{sigmat} and \eqref{ctdef}, and let
$\wh{\phi}_{m,\apar,\bpar}$ be defined by \eqref{pha}.
Suppose there are real-valued functions $\Tfun_m\in\calC^{1,1}(\C)$
which are locally constant on $\C\setminus\setS_\tau$ such that
$\wh{\phi}_{m,\apar,\bpar}+\Tfun_m$ is strictly subharmonic on $\C$
and that for some
positive number $\kappa<1$ we have
\begin{equation*}
\frac {|\dbar \Tfun_m|^2}
{\Delta\wh{\phi}_{m,\apar,\bpar}+\Delta\Tfun_m}\le\frac{\kappa^2}
{\e^{\apar q_\tau}} \qquad \text{a.e.   on}\quad \C.
\end{equation*}
Suppose furthermore that $f\in\coity(\C)$ satisfies
\begin{equation*}\supp f\subset
\setS_\tau.\end{equation*} Then, if $n\ge ](m-\apar)\tau+\bpar[>0$,
we have that $u_*$, the $L^2_{mQ,n}$-minimal solution to $\dbar
u=f$, satisfies
\begin{equation*}\int_\C\babs{u_*}^2
\e^{\Tfun_m-mQ}\dA\le \frac{\e^{\apar
q_\tau}}{(1-\kappa)^2}\int_{\C}\babs{f}^2 \frac {\e^{\Tfun_m-mQ}}{
(m-\apar)\Delta Q+\Delta\Tfun_m+\bpar c_\tau} \dA.\end{equation*}
\end{thm}

\begin{proof} All
conditions in Th. \ref{strb} are fulfilled
with $\phi=mQ$, $\wh{\phi}=\wh{\phi}_{m,\apar,\bpar}$,
$\Tfun=\Tfun_m$, $\alpha=0$ and $\beta=\apar q_\tau$. In view of
\eqref{bye}, it yields that
\begin{equation*}\label{whop3}\begin{split}\int_\C\babs{u_*}^2\e^{\Tfun_m-mQ}\dA&\le
\frac {\e^{\apar q_\tau}} {(1-\kappa)^2}\int_{\C}\babs{f}^2\frac
{\e^{\Tfun_m-mQ}} {\Delta\wh{\phi}_{m,\apar,\bpar}+
\Delta\Tfun_m}\dA\le\\
&\le \frac {\e^{\apar q_\tau}}
{(1-\kappa)^2}\int_{\C}\babs{f}^2\frac {\e^{\Tfun_m-mQ}}
{(m-\apar)\Delta\wh{Q}_\tau+\Delta\Tfun_m+\bpar c_\tau}\dA.\\
\end{split}\end{equation*} The assertion is now immediate from \eqref{whop3},
since $\Delta Q=\Delta \wh{Q}_\tau$ a.e. on $\supp f$.
\end{proof}

\noindent In Sect. \ref{point}, we shall use the full force of Th.
\ref{boh}. As for now, we just mention the following simple
consequence, which also can be proved easily by more elementary
means, cf. \cite{B}, p. 10.

\begin{cor}\label{bh} Let $Q\in \calC^{1,1}(\C)$ and $Q\ge 1$ on
$\C$. Further, let $\setSp$ be a compact subset of $\setS_\tau$. Put
\begin{equation*}a=\essinf\{\Delta Q(z);\,
z\in\setSp\}.\end{equation*} Let $m_0=\max\{2\apar,(1+\apar)/\tau\}$
and assume that $f\in \coity(\C)$ satisfies \begin{equation*}\supp
f\subset\setSp,\end{equation*} and that $n\ge ](m-\apar)\tau+\bpar[$
and $m\ge m_0$. Then $u_*$, the $L^2_{mQ,n}$-minimal solution to
$\dbar u=f$ satisfies
\begin{equation*}\|u_*\|_{mQ}^2\le
 \frac {2\e^{\apar q_\tau}}
{am+\bpar c_\tau}\|f\|_{mQ}^2.
\end{equation*}
\end{cor}

\begin{proof}
Take $\Tfun_m=0$ in Th. \ref{boh} and observe that $m-\apar\ge m/2$
and also $](m-\apar)\tau+\bpar[>0$ whenever $m\ge m_0$. (Any other
$m_0$ having these properties would work as well, of course.)
\end{proof}

\section{Approximate Bergman projections}

\subsection{Preliminaries}  In this section we state and prove a result (Th. \ref{mlem} below), which we
will use to prove the asymptotic expansion in Th. \ref{th3} in the
next section.
The result in question is a modified version of [Berman et al.
\cite{BBS}, Prop. 2.5, p. 9]. Our proof is elementary but rather
lengthy, and the reader may find it worthwhile to look at the result
and go to the next section at the first read.

\subsection{A convention}
It will be convenient to be able to define integrals $\int_S
\chi(w)Y(w)\diff\mu(w)$ where $S$ is a $\mu$-measurable subset of
$\C$, $\chi$ a cut-off function, and $Y$ a complex-valued
$\mu$-measurable function which is well-defined on $\supp\chi$, but
not necessarily on the entire domain $S$. By convention, we extend
the integrand $Y(w)\chi(w)$ to $\C$ by defining $\chi(w)Y(w)=0$
whenever $\chi(w)=0$, i.e., \textit{we define}
\begin{equation*}\int_S
\chi(w)Y(w)\diff\mu(w):=\int_{S\cap
\supp\chi}\chi(w)Y(w)\diff\mu(w).\end{equation*}

\subsection{Statement of the result} Recall the definition of the set $\setX=\{\Delta Q>0\}$,
as well as of the holomorphic extension $\Qext$ of $Q$ from the
anti-diagonal, see \eqref{psi2}.
For the approximating kernel $K_m^1$, see Definition \ref{def1}. By
$\d S$ we mean the positively oriented boundary of a compact set
$S$. We have the following theorem.

\begin{thm} \label{mlem} Let $Q$ be real-analytic in $\C$. Let $S$ be a compact
subset of $\setX$ and fix a point $z_0\in S$. Then there exists

\begin{enumerate}
\item[(1)] a number $\eps>0$ small enough that $K_m^1(z,w)$ makes
sense and is Hermitian for all $z,w\in\D(z_0;2\eps)$ and all $z_0\in
S$,

\item[(2)] a function $\chi\in\coity(\C)$ with $\chi=1$ in
$\D(z_0;3\eps/2)$, $\chi=0$ outside $\D(z_0;2\eps)$, and
$\|\dbar\chi\|_{L^2}\le 3,$

\item[(3)] a real-analytic function $\nu_0(z,w)$ defined for
$z,w\in\D(z_0;2\eps)$,
\end{enumerate}
such that with
\begin{equation}\label{izm}\Iop_{m}u(z)=\int_{S}
u(w)\chi(w)K_m^1(z,w)\e^{-mQ(w)}\dA(w),\qquad u\in A^2_{mQ},\quad
z\in\D(z_0;\eps)\cap S^\circ,
\end{equation} we have
\begin{equation*}\begin{split}
&\babs{\Iop_{m}u(z)-u(z)-\frac 1 {2\pi\imag}\int_{\d
S}u(w)\chi(w)\e^{m(\Qext(z,\bar{w})-Q(w))}\lpar \frac 1
{z-w}+\nu_0(z,w)\rpar\diff
w} \le\frac{C}{m^{3/2}}\e^{mQ(z)/2}\|u\|_{mQ},\\
\end{split}\end{equation*} for all $u$ and $z$ as above, with a number $C$ which only depends on $\eps$ and
$z_0$. Moreover, the numbers $C$ and $\eps$ may be chosen
independently of $z_0$ for $z_0\in S.$

In particular, if $z_0\in S^\circ$ and if $\eps$ is small enough
that $\D(z_0;2\eps)\subset S^\circ$, then $\chi=0$ on $\d S$ and so
\begin{equation}\label{ching}\babs{\Iop_{m}u(z)-u(z)}\le
Cm^{-3/2}\e^{mQ(z)/2}\|u\|_{mQ}, \qquad u\in A^2_{mQ},\quad
z\in\D(z_0;\eps).\end{equation}
\end{thm}

\begin{rem} The function $\nu_0$ is constructed in the
proof, and it depends only on the function $\Qext$ and its
derivatives. See \eqref{nudef} below.

In the case when $z_0$ is in $S^\circ$ and $\eps$ is small enough
that $\D(z_0;2\eps)\subset S^\circ$, the conclusion \eqref{ching}
makes it seem reasonable to call the functional $u\mapsto \Iop_m
u(z)$ an \textit{approximate Bergman projection}, at least locally,
for $z$ in a neighbourhood of $z_0$.
\end{rem}

\subsection{A rough outline; previous work.} In the proof we shall
construct three operators $\Iop_m^j$, $j=1,2,3$, such that $\Iop_m
u(z)=\Iop_m^1 u(z)+\Iop_m^2 u(z)+\Iop_m^3 u(z)$ for $z$ near $z_0$.
The constructions will be made so that
\begin{equation*}\begin{split}&\Iop_m^1 u(z)\quad\text{is close to}\quad  u(z)+\frac 1
{2\pi\imag}\int_{\d
S}\frac{u(w)\chi(w)\e^{m(\Qext(z,\bar{w})-Q(w))}} {z-w}\diff
w,\\
&\Iop_m^2 u(z)\quad \text{is close to}\quad \frac 1
{2\pi\imag}\int_{\d
S}u(w)\chi(w)\nu_0(z,w)\e^{m(\Qext(z,\bar{w})-Q(w))}\diff
w,\\
&\text{and}\quad \Iop_m^3 u(z)\quad \text{is "negligible" for
large\,} m.\\
\end{split}\end{equation*} The operator $\Iop_m^1$ is easy to construct
directly, whereas the definitions of the other operators $\Iop_m^j$
for $j=2,3$ are somewhat more involved. We have found it convenient
to start by considering $\Iop_m^1$.

Our approach is based on the paper \cite{BBS} by Berman et al., in
which a somewhat different situation is treated (notably, $Q$ is
assumed strictly subharmonic in the entire plane in \cite{BBS},
which means that a somewhat different inner product $|u|^2_{mQ}=\int
\babs{u}^2\e^{-mQ}\Delta Q\dA$ is available).

We have found the remarks of Berman in \cite{B}, p. 9 quite useful;
the construction in \cite{BBS} is \textit{local} and hence the
requirement that $Q$ be globally strictly subharmonic can, at least
in principle, be removed. We have found it motivated to give a
detailed proof for this statement.

Our approach is heavily inspired by that of the aformentioned
papers, but is frequently more elementary. The rest of this section
is devoted to the proof of Th. \ref{mlem}.

\subsection{The first approximation.} Throughout the proof, we keep
a point $z_0\in S$ arbitrary but fixed, where $S$ is a given compact
subset of $\setX$.
Also fix a function $u\in A^2_{mQ}$.

The idea is to construct a suitable complex phase function
$\phi_z(w)$, such that for fixed $z$ in $S^\circ$ close to $z_0$,
the main contribution in \eqref{izm} is of the form
\begin{equation}\label{grep}\Iop_m^1 u(z)=\int_{S} \frac {u(w)\chi(w)}
{z-w}\dbar_w\lpar \e^{m\phi_z(w)(z-w)}\rpar\dA(w).\end{equation}

To construct $\phi_z$, we first define, for points $z,w,$ and
$\bar{\zeta}$ sufficiently near each other,
\begin{equation}\label{bri}\theta(z,w,\zeta)=\int_0^1
\d_1\Qext(w+t(z-w),\zeta)\diff t.
\end{equation}
Then $\theta$ is holomorphic in a neighbourhood of the subset
$\{(z,z,\bar{z});z\in \C\}\subset \C^3$ and
\begin{equation*}\theta(z,w,\zeta)(z-w)=\Qext(z,\zeta)-\Qext(w,\zeta).
\end{equation*}
Let $\eps>0$ be sufficiently small that $\D(z_0;2\eps)\Subset \setX$
and such that the functions $\Qext(z,w)$, $\bfun_0(z,w)$ and
$\bfun_1(z,w)$ make sense and are holomorphic on
$\{(z,w);z,\bar{w}\in \D(z_0;2\eps)\}$ (see Definition \ref{def1}).
Writing
\begin{equation*}\delta_0=\Delta Q(z_0)/2,\end{equation*}
we may then use \eqref{bbs} to
choose $\eps>0$ somewhat smaller if necessary so that
\begin{equation}\label{goodc}2\re\{\theta(z,w,\bar{w})(z-w)\}=2\re
\Qext(z,\bar{w})-2Q(w)\le
-\delta_0\babs{z-w}^2+Q(z)-Q(w),\end{equation} for all $z,w\in
\D(z_0;2\eps)$. Note that the same $\eps$ can be used for all $z_0$
in $K$.

Next, we fix a cut-off function $\chi\in\coity(\C)$ such that
$\chi=1$ in $\D(z_0;3\eps/2)$, $\chi=0$ outside $\D(z_0;2\eps)$,
$0\le \chi\le 1$ on $\C$ and
$\|\dbar\chi\|_{L^2}\le 3$, see Subsect. \ref{cutoff}.
Take a point $z\in S^\circ\cap\D(z_0;\eps)$. We now put
\begin{equation*}\phi_z(w)=\theta(z,w,\bar{w}),\qquad\text{which means that}
\qquad \phi_z(w)(z-w)=\Qext(z,\bar{w})-Q(w),\end{equation*} and
consider the corresponding integral $\Iop_m^1 u(z)$ given by
\eqref{grep}.

We have the following lemma. The result should be compared with
\cite{BBS}, Prop. 2.1.

\begin{lem} \label{blem}There exist positive numbers $C_1$ and $\delta$ depending
only on
$z_0$ such that
\begin{equation}\label{goo1}\babs{\Iop_{m}^1 u(z)-\frac 1 {2\pi\imag}\int_{\d S}
\frac {u(w)\chi(w)\e^{m\phi_z(w)(z-w)}} {z-w}\diff w-u(z)}\le
C_1\eps^{-1}\e^{m(Q(z)/2-\delta\eps^2)}\|u\|_{mQ},\end{equation} for all $u\in
A^2_{mQ}$ and all $z\in\D(z_0;\eps)\cap S^\circ$. The numbers $C_1$
and $\delta$ may also be chosen independently of the point $z_0\in
S$. (Indeed, one may take $\delta=\Delta Q(z_0)\eps^2/16$ and
$C_1=6/\eps$.)
\end{lem}

\begin{proof} Keep $z\in\D(z_0;\eps)\cap S^\circ$ arbitrary but
fixed and take $u\in A^2_{mQ}$. Since
$\chi(z)=1$ and since $\babs{w-z}\ge\eps/2$ when $\dbar\chi(w)\ne
0$, \eqref{grep} and Cauchy's formula implies (keep in mind that
$z\in S^\circ$)
\begin{equation}\begin{split}\label{pr}u(z)&+\frac 1 {2\pi\imag}\int_{\d S}
\frac {u(w)\chi(w)\e^{m\phi_z(w)(z-w)}} {z-w}\diff w=\int_{S} \frac
{\dbar_w\lpar
u(w)\chi(w)\e^{m\phi_z(w)(z-w)}\rpar} {z-w}\dA(w)=\\
&=\Iop_m^1 u(z)+\int_{S}\frac {u(w)\dbar\chi(w)} {z-w}
\e^{m\phi_z(w)(z-w)}\dA(w)=\\
&=\Iop_m^1 u(z)+\int_{\{w\in S;\,\babs{w-z}\ge \eps/2\}}\frac
{u(w)\dbar\chi(w)} {z-w}
\e^{m\phi_z(w)(z-w)}\dA(w).\\
\end{split}
\end{equation}
Applying the estimate \eqref{goodc} to the last integral in
\eqref{pr} gives
\begin{equation*}\begin{split}&\babs{
\int_{\{w\in S;\,\babs{w-z}\ge \eps/2\}}\frac {u(w)\dbar\chi(w)}
{z-w} \e^{m\phi_z(w)(z-w)}\dA(w)}\le\\
&\qquad\qquad\qquad\le \int_{\{w\in\C;\, \babs{w-z}\ge
{\eps/2}\}}\frac{\babs{ u(w)\dbar\chi(w)}}
{\babs{z-w}}\e^{m(Q(z)/2-Q(w)/2-\delta_0\babs{z-w}^2/2)}\dA(w)\le\\
&\qquad\qquad\qquad\le
(2/\eps)\e^{m(Q(z)/2-\delta_0\eps^2/8)}\int_\C
\babs{u(w)\dbar\chi(w)}\e^{-mQ(w)/2}\dA(w)\le\\
&\qquad\qquad\qquad\le
(2/\eps)\e^{m(Q(z)/2-\delta_0\eps^2/8)}\|u\|_{mQ}\|\dbar\chi\|_{L^2}
\le (6/\eps)\e^{m(Q(z)/2-\delta_0\eps^2/8)},\\
\end{split}
\end{equation*}
where we have used
that $\|\dbar\chi\|_{L^2}\le 3$ to get the last inequality.
Recalling that $\delta_0=\Delta Q(z_0)/2$ and combining with
\eqref{pr}, we get an estimate
\begin{equation*}\babs{u(z)+\frac 1 {2\pi\imag}\int_{\d S}
\frac {u(w)\chi(w)\e^{m\phi_z(w)(z-w)}} {z-w}\diff w-\Iop_m^1
u(z)}\le (6/\eps)\e^{m(Q(z)/2-\Delta
Q(z_0)\eps^2/16)}.\end{equation*} The proof is finished.
\end{proof}

\subsection{The change of variables.}
Keep $z_0\in S$ fixed and note that the definition of $\theta$ (see
\eqref{bri}) implies that
\begin{equation}\label{d3}[\d_3\theta](z,z,\bar{z})=
[\d_1\d_2\Qext](z,\bar{z})=\Delta
Q(z),\end{equation} for all $z\in \C$. Since $\Delta Q(z_0)>0$,
\eqref{d3} implies that there exists a neighbourhood $U$ of
$(z,w,\zeta)=(z_0,z_0,\bar{z}_0)$ such that the function
\begin{equation*}F:U\to F(U)\quad:\quad
(z,w,\zeta)\mapsto (z,w,\theta(z,w,\zeta))\end{equation*} defines a
biholomorphic change of coordinates. In particular, we may regard
$\zeta$ as a function of the parameters $z$, $w$ and $\theta$ when
$(z,w,\zeta)$ is in $U$.

In view of \eqref{d3} we can choose $\eps>0$ such that
$[\d_3\theta](z,w,\zeta)\ne 0$ whenever
$z,w,\bar{\zeta}\in\D(z_0;2\eps)$. Since
$\phi_z(w)=\theta(z,w,\bar{w})$, we have
\begin{equation}\label{basic}
[\d_3\theta](z,w,\bar{w})=\dbar\phi_z(w),\end{equation} and we infer
that when $z$ is fixed in $\D(z_0;\eps)$ we have
\begin{equation*}\dbar \phi_z(w)\ne 0\quad \text{when}\quad w\in\D(z_0;2\eps),
\end{equation*}

Moreover, choosing $\eps>0$ somewhat smaller if necessary, we can
assume that $U=\{(z,w,\zeta);z,w,\bar{\zeta}\in\D(z_0;2\eps)\}$.
Again, note that the same $\eps$ can be used for all $z_0$ in $S$.

We shall in the following fix a number $\eps>0$ with the above
properties, in addition to the properties which we required earlier,
i.e., $\D(z_0;2\eps)\Subset \setX$, the inequality \eqref{goodc}
holds for all $z,w\in\D(z_0;2\eps)$, and the functions $\Qext$,
$\bfun_0$ and $\bfun_1$ are holomorphic in the set
$\{(z,w);z,\bar{w}\in\D(z_0;2\eps)\}$.

\subsection{The functions $\Theta_0$ and $\Upsilon_0$ and a formula for
$\Iop_m$.} We now use the correspondence $F$ to define a holomorphic
function $\wt{\Qext}$ in $\D(z_0;2\eps)\times F(U)$ via
\begin{equation*}\wt{\Qext}(z,\xi,w,\theta)=\Qext(z,\zeta(\xi,w,\theta)).\end{equation*}
We also put
\begin{equation}\label{d0}\Theta_0(z,w,\theta)=[\d_1\d_4
\wt{\Qext}](z,z,w,\theta)=\bfun_0(z,\zeta)\cdot[\d_\theta\zeta](z,w,\theta)=
\bfun_0(z,\zeta)/[\d_\zeta\theta](z,w,\zeta),\quad
\zeta=\zeta(z,w,\theta),\end{equation}
where we recall that
\begin{equation*}\bfun_0=\d_1\d_2\Qext.\end{equation*}
The function $\Theta_0$ was put to use by Berman et al. in
[\cite{BBS}, p. 9], where it
is called $\Delta_0$.

Note that \eqref{bri} implies that
$\theta(z,z,\zeta)=\d_1\Qext(z,\zeta)$ for all $z$ and $\zeta$, so
that, $[\d_\zeta\theta](z,z,\zeta)=\bfun_0(z,\zeta)$. Hence
\eqref{d0} gives
\begin{equation}\label{d01}\Theta_0(z,z,\theta)=
\bfun_0(z,\zeta)/\bfun_0(z,\zeta)=1.\end{equation} Thus putting
\begin{equation*}\Upsilon_0(z,w,\theta)=-\int_0^1[\d_2\Theta_0](z,z+t(w-z),\theta)\diff
t,\end{equation*} we obtain the relation
\begin{equation}\label{upsrel}\Upsilon_0(z,w,\theta)(z-w)=\Theta_0(z,w,\theta)-1.\end{equation}

We have the following lemma.

\begin{lem}\label{funs}
\begin{equation}\label{boss}\Iop_{m} u(z)=\int_{S} \frac
{u(w)\chi(w)}
 {z-w}\lpar 1+\frac {\bfun_1(z,\bar{w})}
{m\bfun_0(z,\bar{w})}\rpar \Theta_0(z,w,\phi_z(w))\dbar_w\lpar
\e^{m\phi_z(w)(z-w)}\rpar \dA(w).\end{equation}
\end{lem}

\begin{proof}
Combining \eqref{d0} with \eqref{basic} we get the identity
\begin{equation}\label{neq}
\Theta_0(z,w,\phi_z(w))=\bfun_0(z,\bar{w})/\dbar\phi_z(w).\end{equation}
Recalling that
$K_m^1(z,w)=(m\bfun_0(z,\bar{w})+\bfun_1(z,\bar{w}))\e^{m\Qext(z,\bar{w})}$
and $\phi_z(w)=\theta(z,w,\bar{w})$ we thus have (see \eqref{izm})
\begin{equation*}\begin{split}\Iop_{m} u(z)&=\int_{S}
u(w)\chi(w)K_m^1(z,\bar{w})\e^{-mQ(w)}\dA(w)=\\
&=\int_{S}
u(w)\chi(w)(m\bfun_0(z,\bar{w})+\bfun_1(z,\bar{w}))\e^{m\phi_z(w)(z-w)}
\dA(w)=\\
&=\int_{S} u(w)\chi(w) (m\bfun_0(z,\bar{w})+\bfun_1(z,\bar{w}))\frac
{\dbar_w (\e^{m\phi_z(w)(z-w)})} {m(z-w)\dbar \phi_z(w)}\dA(w)=\\
&=\int_{S} \frac {u(w)\chi(w)}
 {z-w}\lpar 1+\frac {\bfun_1(z,\bar{w})}
{m\bfun_0(z,\bar{w})}\rpar \frac{\bfun_0(z,\bar{w})} {\dbar
\phi_z(w)}
\dbar_w\lpar \e^{m\phi_z(w)(z-w)}\rpar \dA(w)=\\
&=\int_{S} \frac {u(w)\chi(w)}
 {z-w}\lpar 1+\frac {\bfun_1(z,\bar{w})}
{m\bfun_0(z,\bar{w})}\rpar \Theta_0(z,w,\phi_z(w))\dbar_w\lpar
\e^{m\phi_z(w)(z-w)}\rpar \dA(w),
\end{split}\end{equation*}
where we have used \eqref{neq} to get the last identity.
\end{proof}

Comparing \eqref{boss} with the formula \eqref{grep}, we now see a
connection between $\Iop_m u(z)$ and $\Iop_m^1 u(z)$. We shall
exploit this relation within short.

\subsection{The differential equation.} We will now show that the
functions $\bfun_0$, $\bfun_1$ and $\Theta_0$ are connected via an
important differential equation.

\begin{lem} \label{comp} For all $z,\bar{\zeta}\in\D(z_0;2\eps)$, we have the
identity
\begin{equation}\label{b1}\frac {\bfun_1(z,\zeta)}
{\bfun_0(z,\zeta)}=-\lpar\frac {\d^2} {\d w\d\theta}
\Theta_0(z,w,\theta)\rpar\biggm|_{z=w,\, \theta=\theta(z,z,\zeta)}.
\end{equation}
\end{lem}

\begin{proof}
In view of \eqref{funrel}, it suffices to show that the holomorphic
function
\begin{equation}\label{B1}B_1(z,\zeta)=-\bfun_0(z,\zeta)\cdot
[\d_2\d_3\Theta_0](z,z,\theta(z,z,\zeta)))\end{equation} satisfies
$B_1(z,\zeta)=\frac 1 2 \frac \d {\d \zeta}\lpar \frac 1
{\bfun_0(z,\zeta)}\frac \d {\d z}\bfun_0(z,\zeta)\rpar$ for all $z$
and $\bar{\zeta}$ in a neighbourhood of $z_0$.

To this end, we first note that a Taylor expansion yields that
\begin{equation*}\Qext(w,\zeta)=\Qext(z,\zeta)+ (w-z)\frac
{\d \Qext} {\d z}(z,\zeta)+\frac {(w-z)^2} 2 \frac {\d^2\Qext} {\d
z^2}(z,\zeta)+\calO ((w-z)^3),\quad \text{as } w\to
z,\end{equation*} i.e.,
\begin{equation*}\theta(z,w,\zeta)=\frac {\Qext(w,\zeta)-\Qext(z,\zeta)}{w-z}=
\frac {\d \Qext} {\d z}(z,\zeta)+\frac {w-z} 2 \frac {\d^2 \Qext}
{\d z^2}(z,\zeta)+\calO((w-z)^2).
\end{equation*}
Differentiating with respect to $\zeta$ and using that
$\bfun_0=\d_1\d_2\Qext$ yields
\begin{equation}\label{burt}\frac {\d\theta} {\d\zeta}=\bfun_0(z,\zeta)+
\frac {w-z} 2 \frac {\d \bfun_0} {\d
z}(z,\zeta)+\calO((w-z)^2).\end{equation} Dividing both sides of
\eqref{burt} by $\bfun_0=\bfun_0(z,\zeta)$ and using \eqref{d0}, we
get
\begin{equation*}\frac 1 {\Theta_0}=
\frac 1 {\bfun_0}\frac {\d\theta}{\d \zeta}=
1+\frac {w-z} {2 \bfun_0}\frac {\d \bfun_0} {\d
z}+\calO((w-z)^2).\end{equation*} Inverting this relation we obtain
\begin{equation*}\label{doeq}\Theta_0=\bfun_0/\frac {\d\theta}{\d\zeta}=
1-\frac {w-z} {2\bfun_0}\frac {\d \bfun_0} {\d
z}+\calO((w-z)^2).\end{equation*} In view of \eqref{B1}, it yields
that
\begin{equation*}\begin{split}B_1(z,\zeta(z,z,\theta))
&=-\bfun_0(z,\zeta)\cdot \frac {\d^2} {\d\theta\d w}
\Theta_0(z,w,\theta)\biggm|_{w=z}=\\
&=\frac {\bfun_0(z,\zeta)} 2\frac \d {\d\theta}\lpar \frac 1
{\bfun_0(z,\zeta)}\frac \d {\d z}
\bfun_0(z,\zeta)\rpar\biggm|_{z=w}=\\
&=\frac {\bfun_0(z,\zeta)} 2 \frac \d {\d\zeta}\lpar \frac 1
{\bfun_0(z,\zeta)}\frac\d {\d z} \bfun_0(z,\zeta)\rpar\biggm|_{z=w}/
\frac
{\d\theta}{\d\zeta}(z,z,\zeta)=\\
&=\frac 1 2 \frac \d {\d\zeta}\lpar \frac 1
{\bfun_0(z,\zeta)}\frac\d {\d z} \bfun_0(z,\zeta)\rpar=\bfun_1(z,\zeta).\\
\end{split}
\end{equation*}
Here we have used \eqref{d0}, \eqref{d01} and \eqref{funrel} to
obtain the last equality.
\end{proof}

\subsection{The operator $\Dop_m$.} For the further analysis, it
will be convenient to consider a differential operator $\Dop_m$
defined for complex smooth functions $A$ of the parameters $z$, $w$
and $\theta$ by
\begin{equation}\label{nabla}\Dop_m A=\frac {\d A} {\d\theta}+m(z-w)A=
\e^{-m\theta(z-w)}\frac \d {\d \theta}\lpar
\e^{m\theta(z-w)}A(z,w,\theta)\rpar.\end{equation} Cf. \cite{BBS},
p. 5.  It will be useful to note that, evaluating at a point where
$\theta=\phi_z(w)$ we have
\begin{equation}\label{eval}\Dop_m A\biggm|_{\theta=\phi_z(w)}=\frac {\e^{-m\phi_z(w)(z-w)}}
{\dbar \phi_z(w)}\dbar_w\lpar
\e^{m\phi_z(w)(z-w)}A(z,w,\phi_z(w))\rpar,\end{equation} as is
easily verified by carrying out the differentiation and comparing
with \eqref{nabla}. From this, we now derive an important identity,
\begin{equation}\begin{split}\label{impo}\int_{S} u(w)&\chi(w)[\Dop_m
A](z,w,\phi_z(w))\dbar_w\lpar\e^{m\phi_z(w)(z-w)}\rpar\dA(w)=\\
&=m\int_{S} u(w)\chi(w)\dbar_w\lpar
\e^{m\phi_z(w)(z-w)}A(z,w,\phi_z(w))\rpar\dA(w)=\\
&=-m\int_{S} u(w)\dbar\chi(w)
\e^{m\phi_z(w)(z-w)}A(z,w,\phi_z(w))\dA(w)+\\
&\quad +\frac m {2\pi\imag}\int_{\d
S}u(w)\chi(w)\e^{m\phi_z(w)(z-w)}A(z,w,\phi_z(w))\diff
w=\\
&=-m\int_{\{w\in S;\,\babs{z-w}\ge \eps/2\}} u(w)\dbar\chi(w)
\e^{m\phi_z(w)(z-w)}A(z,w,\phi_z(w))\dA(w)+\\
&\quad +\frac m {2\pi\imag}\int_{\d
S}u(w)\chi(w)\e^{m\phi_z(w)(z-w)}A(z,w,\phi_z(w))\diff w ,\qquad
z\in\D(z_0;\eps)
.\\
\end{split}
\end{equation}
Here we have used \eqref{eval} to deduce the first equality. The
second equality follows from Green's formula. To get the last
equality we have used that $\babs{z-w}\ge \eps/2$ when
$\dbar\chi(w)\ne 0$.

 We now have the
following lemma, which is based on [\cite{BBS}, Lemma 2.3 and the
discussion on p. 9].

\begin{lem} \label{tlem} There exist holomorphic functions
$A_m(z,w,\theta)$ and $B(z,w,\theta)$, uniformly bounded on
$F(U)$, such that
\begin{equation}\label{tobe}\lpar 1+
\frac{\bfun_1(z,\zeta)}{m\bfun_0(z,\zeta)}\rpar
\Theta_0(z,w,\theta)=1+m^{-1}\Dop_m
A_m(z,w,\theta)+m^{-2}B(z,w,\theta),\end{equation} where it is
understood that $\theta=\theta(z,w,\zeta)$. Moreover, we have that
\begin{equation}\label{alim}\lim_{m\to\infty}A_m=\Upsilon_0,\end{equation}
with uniform convergence on $F(U)$.
\end{lem}

\begin{proof} Put
\begin{equation*}\Theta_1(z,w,\theta)=\frac {\bfun_1(z,\zeta)}
{\bfun_0(z,\zeta)}\Theta_0(z,w,\theta),\end{equation*} so that the
left hand side in \eqref{tobe} can be written
\begin{equation}\label{obet}
\lpar 1+\frac {\bfun_1(z,\zeta)}
{m\bfun_0(z,\zeta)}\rpar\Theta_0(z,w,\theta)=
\Theta_0(z,w,\theta)+m^{-1}\Theta_1(z,w,\theta).\end{equation} Next,
recall that $\Theta_0(z,w,\theta)-1=(z-w)\Upsilon_0(z,w,\theta),$
cf. \eqref{upsrel}. Further, from the relation \eqref{b1}, which
reads
\begin{equation*}\frac {\bfun_1(z,\zeta)} {\bfun_0(z,\zeta)}=-\lpar \frac
{\d^2} {\d\theta\d w}\Theta_0(z,w,\theta)\rpar\biggm|_{z=w},\quad
\zeta=\zeta(z,z,\theta),\end{equation*} it follows that
\begin{equation}\begin{split}\label{1029}\Theta_1(z,z,\theta)&=-\lpar
\frac {\d^2}  {\d\theta\d w}
 \Theta_0(z,w,\theta)\rpar \biggm|_{z=w}\cdot \Theta_0(z,z,\theta)
= -\lpar \frac {\d^2} {\d\theta\d w} \Theta_0(z,w,\theta)\rpar \biggm|_{z=w},\\
\end{split}\end{equation}
where we have used that $\Theta_0(z,z,\theta)=1$ (see \eqref{d01})
to get the last equality. \eqref{1029} allows us to write
\begin{equation*}\Theta_1(z,w,\theta)+\d_\theta\d_w\Theta_0(z,w,\theta)=(z-w)\Upsilon_1(z,w,\theta),\end{equation*}
where $\Upsilon_1$ is holomorphic. Now define
\begin{equation*}A_m=\Upsilon_0+m^{-1}(\Upsilon_1-(\d_\theta\d_w \Upsilon_0))\quad, \quad
B=-\d_\theta(\Upsilon_1-(\d_\theta\d_w)\Upsilon_0).\end{equation*}
Then $A_m$ and $B$ are holomorphic and $A_m\to\Upsilon_0$ as
$m\to\infty$.

We will see that straightforward calculations show that
\begin{equation}\label{dopf}\Dop_m A_m=m\lpar \Theta_0-1\rpar+\Theta_1+m^{-1}\frac \d
{\d\theta}\lpar \frac {\Theta_1} {z-w}-\frac {\d_\theta \Theta_0}
{(z-w)^2}\rpar=m\lpar
\Theta_0-1\rpar+\Theta_1-m^{-1}B.\end{equation} Indeed, when $z\ne
w$ we have
\begin{equation*}\begin{split}A_m&=\frac {\Theta_0-1} {z-w}+m^{-1}\lpar \frac
{\Theta_1+\d_\theta\d_w\Theta_0} {z-w}-\d_w\lpar\frac
{\d_\theta\Theta_0} {z-w}\rpar\rpar=\\
&=\frac {\Theta_0-1} {z-w}+m^{-1}\lpar \frac {\Theta_1} {z-w}-\frac
{\d_\theta\Theta_0} {(z-w)^2}\rpar\\
\end{split}\end{equation*}
so that
\begin{equation*}\d_\theta A_m=\frac {\d_\theta \Theta_0}
{z-w}+m^{-1}\lpar \frac {\d_\theta\Theta_1} {z-w}-\frac
{\d_\theta^2\Theta_0} {(z-w)^2}\rpar,\end{equation*} and
\begin{equation*}m(z-w)A_m=m\lpar \Theta_0-1\rpar+\Theta_1-\frac
{\d_\theta\Theta_0} {z-w},\end{equation*} which gives \eqref{dopf},
since $\Dop_m A_m=\d_\theta A_m+m(z-w)A_m$.

Dividing through by $m$ in the \eqref{dopf} now gives
\begin{equation*}\Theta_0+m^{-1}\Theta_1=1+m^{-1}\Dop_m
A_m+m^{-2}B,\end{equation*}
and glancing at \eqref{obet} we obtain
\eqref{tobe}.
\end{proof}

\subsection{Decomposition of $\Iop_m$.} In view of \eqref{boss} and Lemma
\ref{tlem}, we can now write
\begin{equation}\begin{split}\label{recon}\Iop_{m} u(z)&=\int_{S} \frac {u(w)\chi(w)}
 {z-w}\lpar 1+\frac {\bfun_1(z,\bar{w})}
{m\bfun_0(z,\bar{w})}\rpar
\Theta_0(z,w,\bar{w})\dbar_w\lpar \e^{m\phi_z(w)(z-w)}\rpar \dA(w)=\\
&=\int_{S} \frac {u(w)\chi(w)} {z-w}
\dbar_w(\e^{m\phi_z(w)(z-w)})\dA(w)+\\
&+\frac 1  m\int_{S} \frac {u(w)\chi(w)} {z-w}[\Dop_m
A_m](z,w,\phi_z(w))\dbar_w(\e^{m\phi_z(w)(z-w)})\dA(w)+\\
&+\frac 1 {m^2}\int_{S} \frac
{u(w)\chi(w)} {z-w}B(z,w,\phi_z(w))\dbar_w(\e^{m\phi_z(w)(z-w)})\dA(w),\\
\end{split}\end{equation}
where $A_m$ and $B$ are uniformly bounded and holomorphic near
$(z_0,z_0,\phi_{z_0}(z_0))$.

We recognize the first integral in the right hand side of
\eqref{recon} as $\Iop_{m}^1 u(z)$ (see \eqref{grep}). Let us denote
the others by
\begin{equation*}\Iop_{m}^2 u(z)=\frac 1 m \int_{S} \frac {u(w)\chi(w)}
{z-w}\Dop_m
A_m(z,w,\phi_z(w))\dbar_w(\e^{m\phi_z(w)(z-w)})\dA(w),\end{equation*}
and
\begin{equation}\label{iih}
\Iop_{m}^3 u(z)=\frac 1 {m^2} \int_{S} \frac {u(w)\chi(w)}
{z-w}B(z,w,\phi_z(w))\dbar_w(\e^{m\phi_z(w)(z-w)})\dA(w).\end{equation}

\subsection{Estimates for $\Iop_m^2$.} We start by estimating
$\Iop_{m}^2 u(z)$ when $z\in\D(z_0;\eps)$, and note that
\eqref{impo} yields that
\begin{equation}\begin{split}\label{ffsp}\Iop_{m}^2 u(z)=&-\int_{\{w\in S;\,\babs{z-w}\ge\eps/2\}} u(w)\dbar \chi(w)
\e^{m\phi_z(w)(z-w)}A_m(z,w,\phi_z(w))\dA(w)+\\
&+\frac 1 {2\pi\imag}\int_{\d
S}u(w)\chi(w)\e^{m\phi_z(w)(z-w)}A_m(z,w,\phi_z(w))\diff
w.\\
\end{split}
\end{equation}
Let us now define
\begin{equation}\label{nudef}\nu_0(z,w)=\Upsilon_0(z,w,\phi_z(w)),\end{equation}
so that $A_m(z,w,\phi_z(w))\to\nu_0(z,w)$ with uniform convergence
for $z,w\in\D(z_0;2\eps)$ when $m\to\infty$, by Lemma \ref{tlem}.
Also let $C'=C'(\eps)$ be an upper bound for
$\{\babs{A_m(z,w,\phi_z(w))};\, z,w\in\D(z_0;2\eps),\, m\ge 1\}$. We
now use \eqref{ffsp} and the estimate \eqref{goodc} to get
\begin{equation}\begin{split}\label{goo2}&\babs{\Iop_{m}^2 u(z)-\frac 1 {2\pi\imag}\int_{\d S}
u(w)\chi(w)\e^{m\phi_z(w)(z-w)}\nu_0(z,w)\diff w}\le\\
&\qquad\le C'\int_{\{w\in S;\,\babs{w-z}\ge \eps/2\}}
\babs{u(w)\dbar\chi(w)}\e^{m(Q(z)/2-Q(w)/2-\delta_0\babs{z-w}^2/2)}\dA(w)\le\\
&\qquad\le C'\e^{m(Q(z)/2-\delta_0\eps^2/8)}\int_\C
\babs{u(w)\dbar\chi(w)}\e^{-Q(w)/2}\dA(w)\le\\
&\qquad\le
C'\e^{m(Q(z)/2-\delta_0\eps^2/8)}\|u\|_{mQ}\|\dbar\chi\|_{L^2}\le
C_2\e^{m(Q(z)/2-\delta\eps^2)}\|u\|_{mQ},\\
\end{split}
\end{equation}
where $\delta=\delta_0/8=\Delta Q(z_0)/16$, and
$C_2=3C'$.

\subsection{Estimates for $\Iop_m^3$.} To estimate
$\Iop_m^3 u(z)$, we note that \eqref{iih} implies
\begin{equation*}\Iop_m^3 u(z)=\frac 1 m \int_\C u(w)\chi(w)B(z,w,\phi_z(w))
\e^{m\phi_z(w)(z-w)}\dbar\phi_z(w)\dA(w).\end{equation*} It suffices
to estimate this integral using \eqref{goodc}. This gives
\begin{equation*}\babs{\Iop^3_{m} u(z)}\le \frac 1 m \int_\C
\babs{u(w)\chi(w)B(z,w,\phi_z(w))}\e^{m(Q(z)/2-Q(w)/2)-m\delta_0\babs{z-w}^2/2}\babs{\dbar
\phi_z(w)}\dA(w).\end{equation*} Since the function
$\babs{B(z,w,\phi_z(w))\dbar\phi_z(w)\chi(w)}$ can be estimated by a number
$C_3$ independent of $z$ and $w$ for $z\in\D(z_0;\eps)$ and $w\in
\D(z_0;2\eps)$, it gives, in view of the Cauchy--Schwarz inequality
\begin{equation}\label{goo3}\begin{split}\babs{\Iop^3_{m}u(z)}&\le
C_3m^{-1}\e^{mQ(z)/2}\int_\C\babs{u(w)\chi(w)}\e^{-mQ(w)/2}\dA(w)\le\\
&\le C_3m^{-1}\e^{mQ(z)/2}\|u\|_{mQ}\lpar\int_\C \e^{-m\delta_0\babs{z-w}^2}\dA(w)\rpar^{1/2}\le
C_3\delta_0^{-1/2}m^{-3/2}\e^{mQ(z)/2}\|u\|_{mQ},\\
\end{split}\end{equation} where
$C_3=C_3(\eps)=\sup\big\{\babs{B(z,w,\phi_z(w))\dbar\phi_z(w)\chi(w)};\,
z,w\in \D(z_0;2\eps)\big\}$.

\subsection{Conclusion of the proof of Theorem \ref{mlem}} By
\eqref{recon} and the estimates \eqref{goo1}, \eqref{goo2} and
\eqref{goo3}, (using that $\phi_z(w)(z-w)=\Qext(z,\bar{w})-Q(w)$)
\begin{equation*}\begin{split}&\babs{\Iop_{m} u(z)-u(z)-\frac 1
{2\pi\imag}\int_{\d S}u(w)\chi(w)\lpar \frac 1
{z-w}+\nu_0(z,w)\rpar\e^{m(\Qext(z,\bar{w})-Q(w))}\diff
w}\le\\
&\le \babs{\Iop_{m}^1 u(z)-u(z)-\frac 1 {2\pi\imag}\int_{\d
S}\frac{u(w)\chi(w)\e^{m\phi_z(w)(z-w)}}
 {z-w}\diff w}+\\
&\qquad+\babs{\Iop_{m}^2u(z)-\frac 1 {2\pi\imag}\int_{\d
S}u(w)\chi(w)\nu_0(z,w)\e^{m\phi_z(w)(z-w)}\diff w}+
\babs{\Iop_{m}^3u(z)}\le\\
&\le C_1\eps^{-1}\e^{m(Q(z)/2-\delta\eps^2)}\|u\|_{mQ}+C_2\e^{m(Q(z)/2-\delta\eps^2)}
\|u\|_{mQ}+C_3\delta_0^{-1}m^{-3/2}\e^{mQ(z)/2}\|u\|_{mQ},\quad z\in\D(z_0;\eps)\cap S^\circ.\\
\end{split}
\end{equation*}
It thus suffices to choose a number $C$ large enough that
\begin{equation}\label{itsp}C_1\e^{-\delta m}+C_2\e^{-\delta m}
+C_3m^{-3/2}
 \le Cm^{-3/2},\quad m\ge 1.\end{equation}
Now recall that the numbers $C_1,$ $C_2,$ and $C_3$ only depend on
$\eps$, where $\eps>0$ can be chosen independently of $z_0$ for
$z_0$ in the given compact subset $S$ of $\setX$. Moreover, we have
that $\delta=\Delta Q(z_0)\eps^2/16$, and $\Delta Q(z_0)$ is bounded
below by a positive number for $z_0\in S$. It follows that $C$ in
\eqref{itsp} can be chosen independently of $z_0\in S$, and the
proof is finished. \hfill $\qed$

\section{The proof of Theorem \ref{th3}}\label{proof}

\subsection{Preliminaries.}
In this section we prove Th. \ref{th3}. Our approach which follows
[\cite{BBS}, Sect. 3] is obtained by assembling the information from
Lemma \ref{lemm2}, Th. \ref{mlem}, and Cor. \ref{bh}. To facilitate
the presentation, we divide the proof into steps.

First recall that adding a constant to $Q$ means that $K_{m,n}$ is
only changed by a multiplicative constant, and hence we can (and
will) assume that $Q\ge 1$ on $\C$.

Fix a compact subset $K\Subset \setS_\tau^\circ\cap\setX$ and a
point $z_0\in \setS_\tau^\circ\cap \setX$, and let $\eps>0$ and
$\chi\in \coity(\C)$ be as in Th. \ref{mlem}.
Choosing $\eps>0$ somewhat smaller if necessary, we may ensure that
$2\eps<\dist(K,\C\setminus(\setS_\tau\cap\setX))$ and that
\begin{equation}\label{leed}\babs{K_m^1(z,w)}^2\e^{-m(Q(z)+Q(w))}\le
Cm^2\e^{-\delta_0 m\babs{z-w}^2},\quad z,\bar{w}\in
\D(z_0;2\eps),\quad z_0\in K,\end{equation} for some positive
numbers $C$ and $\delta_0$ (see \eqref{lead}). We also introduce the
set
\begin{equation*}\setSp=\{z\in\C;\, \dist(z,K)\le
2\eps\}\Subset \setS_\tau^\circ\cap\setX.\end{equation*} Our goal is
to prove an estimate
\begin{equation*}\babs{K_{m,n}(z,w)-K_m^1(z,w)}\e^{-m(Q(z)+Q(w))/2}\le
Cm^{-1},\quad z,w\in\D(z_0;\eps),\quad n\ge m\tau-M,\quad m\ge m_0,
\end{equation*}
where $M\ge 0$ is given, and $C$ and $m_0$ are independent of the specific
choice of the point $z_0\in K$.

In the following, we let $C$ denote various (more or less) absolute
constants which can change meaning from time to time, even within
the same chain of inequalities.

Let $\Pop_{m,n}:L^2_{mQ}\to \Hspace_{m,n}$ denote the orthogonal
projection, that is,
\begin{equation*}\Pop_{m,n}u(z)=
\int_\C u(w)K_{m,n}(z,w)\e^{-mQ(w)} \dA(w).\end{equation*}
Throughout the proof we \textit{fix} two points $z,w\in\D(z_0;\eps)$
and introduce the functions
\begin{equation}\label{func1}u_z(\zeta)=K_{m,n}(\zeta,z)\quad\text{and}\quad
v_w(\zeta)=\begin{cases} \chi(\zeta)K_m^1(\zeta,w),\quad& \zeta\in
\D(z_0;2\eps),\cr 0, &\text{otherwise},\cr
\end{cases}
\end{equation}
as well as
\begin{equation}\label{dwdef}
d_w(\zeta)=v_w(\zeta)-\Pop_{m,n}v_w(\zeta),\qquad
\zeta\in\C.\end{equation} We note that $v_w\in\calC_0^\infty(\C)$
with $\supp v_w\subset \setSp\Subset\setS_\tau^\circ\cap \setX.$

\begin{lem} \label{lemma2} There exists a positive number $C$ independent
of $m\ge 1$, $n$, $z$ and $w$, such that
\begin{equation}\label{app1}
\big|K_{m,n}(z,w)-\Pop_{m,n}v_w(z)\big|\le  Cm^{-1}
\e^{m(Q(z)+Q(w))/2}.\end{equation} Moreover, $C$ can be chosen
independent of $z_0$ for $z_0\in K$.
\end{lem}

\begin{proof}
By Th. \ref{mlem}, there is a number $C$ depending only on $\eps$
and $K$, such that
\begin{equation*}\babs{\Iop_{m} u_z(w)-u_z(w)}
\le Cm^{-3/2}\e^{mQ(w)/2}\|u_z\|_{mQ},
\end{equation*}
where $\Iop_{m}$ is given by \eqref{izm} with $S=\Sigma$. Next note
that the estimate \eqref{berg1} implies
\begin{equation*}\|u_{z}\|_{mQ}^2=K_{m,n}(z,z)\le
Cm\e^{mQ(z)},
\end{equation*}
with a number $C$ depending only on $\tau$. We conclude that
\begin{equation}\label{fest}\babs{\Iop_{m} u_z(w)-K_{m,n}(w,z)}\le
Cm^{-1}\e^{m(Q(z)+Q(w))/2}.\end{equation}  We next note that
\begin{equation}\label{ik1}\overline{\Iop_{m} u_z(w)}=
\int_\C
\chi(\zeta)K_m^1(\zeta,w)K_{m,n}(z,\zeta)\e^{-mQ(\zeta)}\dA(\zeta)
=\Pop_{m,n}v_w(z),\end{equation} with $v_w$ given by \eqref{func1}.
Thus the statement \eqref{app1} is immediate from \eqref{ik1} and
\eqref{fest}.
\end{proof}

\begin{lem} \label{lemma3}  Let two positive numbers $\apar$ and $\bpar$ be given
such that \eqref{sss} is fulfilled, and put
$m_0=\max\{(1+\apar)/\tau,2\apar\}$. Then for all $m,n$ such that $n\ge
](m-\apar)\tau+\bpar[$ and $m\ge m_0$, there exist positive numbers
$\delta$ and $C$, independent of $z_0$, $m$, $n$, $z$, and $w$, such
that
\begin{equation}\label{sest}
\babs{d_w(z)}^2\le Cm^3\e^{-m\delta}\e^{m(Q(z)+Q(w))}.
\end{equation}
\end{lem}

\begin{proof} The function $d_w$ is the
$L^2_{mQ,n}$-minimal solution to the $\dbar$-equation $\dbar
d_w=\dbar v_w$. Since $\supp v_w\subset \setS_\tau$, Cor. \ref{bh}
yields that
\begin{equation}\label{dbest}
\|d_w\|_{mQ}^2\le C\|\dbar v_w\|_{mQ}^2,\qquad m\ge m_0,\quad n\ge
](m-\apar)\tau+\bpar[,\end{equation} with a number $C$ depending
only on $\tau$, $\apar$ and $\bpar$. But $\dbar v_w(\zeta)=\dbar
\chi(\zeta)K_m^1(\zeta,w)$, whence the estimate \eqref{leed} shows
that there are numbers $C$ and $\delta_0>0$ such that
\begin{equation}\label{test}|\dbar v_w(\zeta)|^2
\e^{-m(Q(\zeta)+Q(w))}\le
Cm^2|\dbar\chi(\zeta)|^2\e^{-m\delta_0\babs{\zeta-w}^2},\qquad
\zeta\in\C.
\end{equation} Now, since $\babs{\zeta-w}\ge \eps/2$ whenever
$\dbar\chi(\zeta)\ne 0$, we deduce from \eqref{test} that
\begin{equation}\label{quest}|\dbar v_w(\zeta)|^2\e^{-mQ(\zeta)}\le
Cm^2|\dbar\chi(\zeta)|^2\e^{m(Q(w)-\delta)},\quad \zeta
\in\C,\end{equation} where $\delta=\delta_0\eps^2/4$. Using
\eqref{dbest} and then integrating the inequality \eqref{quest} with
respect to $\dA(\zeta)$, we get
\begin{equation}\label{cest}\|d_w\|_{mQ}^2\le
C\|\dbar v_w\|_{mQ}^2 \le
Cm^2\e^{m(Q(w)-\delta)}\|\dbar\chi\|_{L^2}^2\le
Cm^2\e^{m(Q(w)-\delta)}.
\end{equation}
Next note that the function $d_w$ is holomorphic in the disk
$\D(z_0;3\eps/2)$, so that Lemma \ref{lemm2} gives
\begin{equation}\label{desto}
|d_w(z)|^2\e^{-mQ(z)}\le
Cm\int_{\D(z;\eps/\sqrt{m})}\babs{d_w(\zeta)}^2\e^{-mQ(\zeta)}\dA(\zeta)\le
Cm \|d_w\|^2_{mQ},
\end{equation}
where $C$ only depends on $\eps$ and $\tau$. Combining \eqref{desto}
with \eqref{cest}, we end up with \eqref{sest}, and so we are done.
\end{proof}

\subsection{Conclusion of the proof of Theorem \ref{th3}.} Since $\chi(z)=1$,
we have that $v_w(z)=K_m^1(z,w)$, whence by \eqref{dwdef}
\begin{equation*}\babs{K_{m,n}(z,w)-K_m^1(z,w)}=
\babs{K_{m,n}(z,w)-v_w(z)}
\le\babs{K_{m,n}(z,w)-\Pop_{m,n}v_w(z)}+\babs{d_w(z)}.
\end{equation*}
In view of Lemmas \ref{lemma2} and \ref{lemma3}, the right hand side
can be estimated by
\begin{equation*}C\lpar m^{-1}+m^{3/2}\e^{-m\delta/2}\rpar
\e^{m(Q(z)+Q(w))/2},\end{equation*} whenever
$n\ge](m\tau-\apar)+\bpar[$ and $m\ge m_0$. For $m\ge 1$, the latter
expression is dominated by $Cm^{-1}\e^{m(Q(z)+Q(w))/2}$. The proof
is finished. \hfill $\qed$

\section{Berezin quantization and Gaussian convergence} \label{p2}

\subsection{Preliminaries.} In this section we use the
expansion formula for $K_{m,n}$ (Th. \ref{th3}) to prove theorems
\ref{th1} and \ref{th2}. In the proofs, we set $\tau=1$ and $m=n$.
(The argument in the general case follows the same pattern.) It then
becomes natural to write $K_n$ for $K_{n,n}$ etc. We also fix a
compact subset $K\Subset \setS_1^\circ\cap\setX$, a point $z_0\in
K$, and a positive number $\eps$ with the properties listed in Th.
\ref{th3}, and we put
\begin{equation*}\delta_n=\eps\log n/\sqrt{n}.\end{equation*}

\subsection{The proof of Theorem \ref{th1}} It suffices to show that
\begin{equation}\label{firstest}B_{n}^{\langle
z_0\rangle}(\D(z_0;\delta_n))\to 1\quad \text{as}\quad
n\to \infty.\end{equation} Indeed, since $B_{n}^{\langle
z_0\rangle}$ is a p.m., this implies Th. \ref{th1}.

In order to prove \eqref{firstest}, we apply Th. \ref{th3}, which
gives
\begin{equation*}K_{n}(z_0,z)\e^{-n(Q(z)+Q(z_0))/2}=
\lpar
n\bfun_0(z_0,\bar{z})+\bfun_1(z_0,\bar{z})\rpar\e^{n(\re\Qext(z_0,\bar{z})-(Q(z)+Q(z_0))/2)}+
\calO(n^{-1}),\quad z\in\D(z_0;\eps)\end{equation*} when
$n\to\infty$, where the $\calO$ is uniform for $z_0\in K$. In view
of \eqref{bbs}, we have
\begin{equation}\label{2ndest}K_{n}(z_0,z)\e^{-n(Q(z)+Q(z_0))/2}=
\lpar n\bfun_0(z_0,\bar{z})+\calO(1)\rpar \e^{n(-\Delta
Q(z_0)/2+R(z_0,z))}+ \calO(n^{-1}),\quad
z\in\D(z_0;\eps)\end{equation} with a function $R$ satisfying
$\babs{R(z,w)}\le C\babs{z-w}^3$ whenever $z,w\in \D(z_0;\eps)$ and
$z_0\in K$. Note that
\begin{equation}\label{mroh}n\babs{R(z_0,z)}\le Cn\delta_n^3\le C\log^3
n/\sqrt{n},\qquad \text{when}\quad z\in\D(z_0;\delta_n),\quad z_0\in
K.\end{equation} Introducing this estimate in \eqref{2ndest} gives
\begin{equation*}\babs{K_{n}(z_0,z)}^2\e^{-n(Q(z)+Q(z_0))}=
\lpar n^2\babs{b_0(z_0,\bar{z})}^2+\calO(n)\rpar\e^{-n\Delta
Q(z_0)\babs{z_0-z}^2+\calO(\log^3 n/\sqrt{n})}+\calO(n^{-2}),
\end{equation*}
for $z\in\D(z_0;\delta_n)$ as
$n\to\infty$, where the $\calO$-terms are uniform for $z_0\in K$.
Furthermore, \eqref{ber} yields that
\begin{equation}\label{3dest}\frac {\babs{K_{n}(z_0,z)}^2}
{K_{n}(z_0,z_0)}\e^{-nQ(z)}=\frac {\babs{K_{n}(z_0,z)}^2} {n\Delta
Q(z_0)+\calO(1)}\e^{-n(Q(z_0)+Q(z))},\quad
z\in\D(z_0;\eps),\end{equation} as
$n\to\infty$. Note that the left hand side in \eqref{3dest} is the
density $\berd_{n}^{\langle z_0\rangle}(z)$, so that \eqref{2ndest}
and \eqref{3dest} implies
\begin{equation}\label{4thest}\berd_{n}^{\langle z_0\rangle}(z)=\frac
{n^2\babs{\bfun_0(z_0,\bar{z})}^2+\calO(n)} {n\Delta
Q(z_0)+\calO(1)}\e^{-n\Delta
Q(z_0)\babs{z-z_0}^2+\calO(\log^3n/\sqrt{n})} +\calO(n^{-2}),\quad
z\in \D(z_0;\delta_n).\end{equation} Integrating \eqref{4thest} with
respect to $\dA$ over $\D(z_0;\delta_n)$ and using the mean-value
theorem for integrals now gives that there are positive numbers
$v_{n,z_0}$ converging to $1$ (uniformly for $z_0\in K$), and also
complex numbers $b_{n,z_0}$ converging to
$\bfun_0(z_0,\bar{z}_0)=\Delta Q(z_0)$ (uniformly for $z_0\in K$) as
$n\to\infty$, such that
\begin{equation*}\begin{split}B_{n}^{\langle z_0\rangle}(\D(z_0;\delta_n))&=
v_{n,z_0}\frac {n^2\babs{b_{n,z_0}}^2+\calO(n)} {n\Delta
Q(z_0)+\calO(1)}\int_{\D(z_0;\delta_n)} \e^{-n\Delta
Q(z_0)\babs{z-z_0}^2}\dA(z)+\calO(n^{-2})\int_{\D(z_0;\delta_n)}\dA(z)=\\
&=v_{n,z_0}\frac {n^2\babs{b_{n,z_0}}^2+\calO(n)} {n^2\Delta
Q(z_0)^2+\calO(n)}\lpar 1-\e^{-\Delta Q(z_0)\eps^2\log^2
n}\rpar+\calO(n^{-2}).\\
\end{split}
\end{equation*}
The expression in the right hand side converges to $1$ as
$n\to\infty$. This proves \eqref{firstest}, and Th. \ref{th1}
follows. \hfill$\qed$

\subsection{The proof of Theorem \ref{th2}}It suffices to to show that
there are numbers $\eps_n$ converging to zero as $n\to\infty$ such
that
\begin{equation}\label{5thest}\int_\C\babs{\berd_{n}^{\langle
z_0\rangle}(z)-n\Delta Q(z_0)\e^{-n\Delta
Q(z_0)\babs{z-z_0}^2}}\dA(z)\le \eps_n,\qquad z_0\in
K.\end{equation} Indeed, \eqref{gc1} follows form \eqref{5thest}
after the change of variables $z\mapsto z_0+z/\sqrt{n\Delta Q(z_0)}$
in the integral in \eqref{5thest}.

To prove \eqref{5thest}, we split the integral with respect to the
decomposition $\{\babs{z-z_0}<\delta _n\}\cup\{\babs{z-z_0}\ge
\delta_n\}$, and put
\begin{equation*}A_{n,z_0}=
\int_{\{z;\babs{z-z_0}<\delta_n\}}\babs{\berd_{n}^{\langle
z_0\rangle}(z)-n\Delta Q(z_0)\e^{-n\Delta
Q(z_0)\babs{z-z_0}^2}}\dA(z),\end{equation*}
\begin{equation*}B_{n,z_0}=
\int_{\{z;\babs{z-z_0}\ge \delta_n\}}\babs{\berd_{n}^{\langle
z_0\rangle}(z)-n\Delta Q(z_0)\e^{-n\Delta
Q(z_0)\babs{z-z_0}^2}}\dA(z).\end{equation*} Considering $A_{n,z_0}$
first, we get from \eqref{5thest} that
\begin{equation*}A_{n,z_0}=
\int_{\D(z_0;\delta_n)}\babs{\e^{nR(z_0,\bar{z})}\frac
{n^2\babs{\bfun_0(z_0,\bar{z})}^2+\calO(n)} {n\Delta
Q(z_0)+\calO(1)}-n\Delta Q(z_0)}\e^{-n\Delta
Q(z_0)\babs{z-z_0}^2}\dA(z)+\calO(n^{-2}),\end{equation*} as
$n\to\infty$. Next we put
\begin{equation*}s_{n,z_0}=\sup_{z\in\D(z_0;\delta_n)}\bigg\{\babs{\e^{nR(z_0,\bar{z})}\frac
{n^2\babs{\bfun_0(z_0,\bar{z})}^2+\calO(n)} {n\Delta
Q(z_0)+\calO(1)}-n\Delta Q(z_0)}\bigg\},\end{equation*} and observe
\eqref{mroh} implies that $s_{n,z_0}/n\to 0$, uniformly for $z_0\in
K$. It yields that
\begin{equation*}A_{n,z_0}\le s_{n,z_0}\int_{\D(z_0;\delta_n)}\e^{-n\Delta
Q(z_0)\babs{z-z_0}^2}\dA(z)+\calO(n^{-2})\le
Cs_{n,z_0}/n,\end{equation*} which converges to $0$ as $n\to \infty$
uniformly for $z_0\in K$.

To estimate $B_{n,z_0}$ we simply observe that
\begin{equation}\label{bmest}B_{n,z_0}\le
\int_{\{z;\babs{z-z_0}\ge\delta_n\}}\berd_{n}^{\langle
z_0\rangle}(z)\dA(z)+\int_{\{z;\babs{z-z_0}\ge\delta_n\}}n\Delta
Q(z_0)\e^{-n\Delta Q(z_0)\babs{z-z_0}^2}\dA(z).
\end{equation}
Since $B_{n}^{\langle z_0\rangle}$ is a p.m., the estimate
\eqref{firstest} yields that the first integral in the right hand
side of \eqref{bmest} converges to $0$ as $n\to\infty$. Moreover, a
simple calculation yields that the second integral in \eqref{bmest}
converges to $0$ when $n\to\infty$. We have shown that $B_{n,z_0}\to
0$ as $n\to \infty$ with uniform convergence for $z_0\in K$. The
proof is finished. \hfill$\qed$

\section{Off-diagonal damping}\label{point}

\subsection{An estimate for $K_{m,n}$.} In this section we prove
Th. \ref{th1.5}. We shall obtain that theorem from Th. \ref{flock}
below, which is of independent interest, and has applications in
random matrix theory, see \cite{AHM}. Our analysis depends on the
following lemma. It will be convenient to define the set
\begin{equation*}\setS_{\tau,1}=\{\zeta;\,
\dist(\zeta,\setS_\tau)\le 1\}.\end{equation*}

\begin{lem}\label{propn1} Assume that $Q\in\calC^2(\C)$.
Let $z_0\in\setS_\tau\cap\setX$ and let $M$ be given non-negative
numbers. Put
\begin{equation*}8d=\dist(z_0,\C\setminus(\setS_\tau\cap\setX))\qquad
a=\inf\{\Delta Q(\zeta);\, \zeta\in\D(z_0;6d)\}\qquad A=\sup\{\Delta
Q(\zeta);\, \zeta\in\setS_{\tau,1}\}.\end{equation*} There then
exists positive numbers $C$ and $\epsilon$ such that for all $m,n\ge
1$ with $m\tau-M\le n\le m\tau+1$, we have
\begin{equation}\label{bupp}
\babs{K_{m,n}(z_0,z)}^2\e^{-m(Q(z_0)+Q(z))}\le C
m^2\e^{-\epsilon\sqrt{m}\min\{d,\babs{z_0-z}\}} ,\qquad z\in
\setS_\tau.
\end{equation}
Here $C$ depends only on $M$, $a$, $A$, and $\tau$, while $\epsilon$
only depends on $a$, $\tau$ and $M$.
\end{lem}

\begin{rem} For fixed
$\tau$ and $M$, our method of proof gives that $\epsilon$ can be
chosen proportional to $a$ while $C$ can be chosen proportional to
$a^{-1}$. Indeed, our proof shows that if there is a positive number
$c$ depending only on $\tau$ and $M$ such that, with $a^\prime=\min\{a,1\}$,
\begin{equation*}\babs{K_{m,n}(z_0,z)}^2\e^{-m(Q(z_0)+Q(z))}\le
C\frac {m^3} {am+c}\e^{-\epsilon'a'\sqrt{m}\min\{d,\babs{z_0-z}\}}
,\qquad z\in \setS_\tau,
\end{equation*}
with $C$ and $\epsilon'$ independent of $a$. This estimate can
easily be extended to all $z\in\C$ by adapting the proof of Th.
\ref{flock} below.
\end{rem}

\noindent Before we turn to the proof of Lemma \ref{propn1}, we use
it to prove the main result of this section.

\begin{thm}\label{flock}
In the situation of Lemma \ref{propn1} we also have an estimate
\begin{equation}\label{brest}\babs{K_{m,n}(z_0,z)}^2\e^{-m(Q(z_0)+Q(z))}\le
Cm^2\e^{-\epsilon\sqrt{m}\min\{d,\babs{z_0-z}\}}\e^{-m
(Q(z)-\wh{Q}_\tau(z))},\qquad z\in\C,\end{equation} valid for all
$z_0\in \setS_\tau\cap X$, $m\ge 1$, and all $n$ with $m\tau-M\le
n\le m\tau+1$.
\end{thm}

\begin{proof} In view of Lemma \ref{propn1}, and since $Q=\wh{Q}_\tau$
on $\setS_\tau$, it suffices to show
the estimate \eqref{brest} when $z\not\in \setS_\tau$. To this end, we
consider the function
\begin{equation*}f(\zeta)=\log\babs{K_{m,n}(z_0,\zeta)}^2-m\wh{Q}_\tau(\zeta).
\end{equation*}
Since $\wh{Q}_\tau$ is harmonic on $\C\setminus\setS_\tau$, $f$ is
subharmonic there. Moreover, since $n-1\le m\tau$ and since
$K_{m,n}(\cdot,z_0)\in \Hspace_{m,n}$, we have a simple estimate
\begin{equation*}\log\babs{K_{m,n}(z_0,\zeta)}^2\le (n-1)\log\babs{\zeta}^2+\calO(1)\le m\tau\log\babs{\zeta}^2+
\calO(1) \qquad \text{when}\quad \zeta\to\infty,\end{equation*}
while the relation \eqref{qtau} says that $\wh{Q}_\tau(\zeta)=
\tau\log\babs{\zeta}^2+\calO(1)$ when $\zeta\to \infty$. Hence $f$
is bounded above. Furthermore, it is clear that $f$ is harmonic in a
punctured neighbourhood of $\infty$, which yields that $f$ has a
representation $f(\zeta)=h(\zeta)-c\log\babs{\zeta}$ for all large
enough $|\zeta|$, where $c$ is a non-negative number and $h$ is
harmonic at $\infty$ (see e.g. [\cite{ST}, Cor. 0.3.7, p. 12]). In
particular, $f$ extends to a subharmonic function on $\C^*\setminus
\setS_\tau$. Finally, since Lemma \ref{propn1} yields that
\begin{equation*}f(\zeta)\le \log(Cm^2)+mQ(z_0)-\epsilon\sqrt{m}d,\qquad
\text{when}\quad \zeta\in\d\setS_\tau,\end{equation*} the maximum
principle shows that the same estimate holds for all $\zeta \in
\C^*\setminus \setS_\tau$. This means that
\begin{equation}\babs{K_{m,n}(z_0,z)}^2\e^{-m\wh{Q}_\tau(z)}\le Cm^2
\e^{-\epsilon\sqrt{m}d}\e^{mQ(z_0)},\qquad \text{when}\quad z\not\in \setS_\tau,
\end{equation}
which implies \eqref{brest} when $z\not\in \setS_\tau$.
\end{proof}

\noindent\bf Background. \rm Estimates related to those considered
in Lemma \ref{propn1} are known, see e.g. Lindholm's article
\cite{L}, Prop. 9, pp. 404--407. We will follow the basic strategy
used in that paper in our proof below. Since we are considering a
different situation with polynomial Bergman kernels (instead of the
full Bergman kernel of $A^2_{mQ}$), and since we are not assuming
the weight $Q$ to be globally strictly subharmonic, nontrivial
modifications of the classical arguments are needed. Our main tool
for accomplishing this is provided by the weighted $L^2$ estimates
in Th. \ref{boh}.

\subsection{The proof of Lemma \ref{propn1}} We can and will
assume that $Q\ge 1$ on $\C$. Moreover, we will denote various
constants by the same letter $C$, which can change meaning during
the course of the calculations. When $C$ depends on one or several
parameters, we will always specify this. To simplify the proof we
will first reduce the problem by treating three simple cases.

\subsection{Case 1: $md^2\le 1$.} Since we have assumed that $n\le m\tau+1$,
the estimate \eqref{berg2} of Prop. \ref{lemma1} applies. It gives
that
\begin{equation}\label{spuc}\babs{K_{m,n}(z_0,z)}^2\e^{-m(Q(z_0)+Q(z))}\le
Cm^2,\quad z\in\C,\end{equation} with a number $C$ depending only on
$\tau$. The assertion \eqref{bupp} follows immediately in case
$d=0$. In the remaining case we have $d>0$ and $m\le d^{-2}$. Then,
for any $\epsilon>0$ and any $z\in\C$, we have that
$\e^{-\epsilon\sqrt{m}\min\{d,\babs{z_0-z}\}}\ge \e^{-\epsilon
d^{-1}d}=\e^{-\epsilon},$ so that
\begin{equation}\label{buss}\babs{K_{m,n}(z_0,z)}^2\e^{-m(Q(z_0)+Q(z))}\le
C\e^\epsilon\e^{-\epsilon\sqrt{m}\min\{d,\babs{z_0-z}\}}.\end{equation}
This gives the desired estimate \eqref{bupp} with $C$ replaced by
$C\e^\epsilon$.

\subsection{Notation}\label{nota}
We now fix positive numbers $\apar$ and $\bpar$ such that the
relation \eqref{sss} is satisfied. We further choose $\apar$ and
$\bpar$ so that $](m-\apar)\tau+\bpar[\le m\tau-M$ for all $m\ge 1$,
and let $m_0=\max\{(1+\apar)/\tau,4\apar\}$. Let $n$ be a positive
integer such that $n\le m\tau+1$. We also fix a point
$z\in\setS_\tau$ and let $R$ be a number such that $\setS_\tau\subset \D(0;R)$.

\subsection{Case 2: $m\le m_0$.}
Given any $\epsilon>0$, we have that
$\e^{-\epsilon\sqrt{m}\min\{\babs{z_0-z},d\}}\ge
\e^{-\epsilon\sqrt{m_0}R},$ and so \eqref{spuc} implies that
\begin{equation*}\babs{K_{m,n}(z_0,z)}^2\e^{-m(Q(z_0)+Q(z))}\le
C\e^{\epsilon\sqrt{m_0}R}\e^{-\epsilon\sqrt{m}\min\{\babs{z_0-z},d\}}.\end{equation*}
Thus \eqref{bupp} holds with $C$ replaced by
$C\e^{\epsilon\sqrt{m_0}R}$.

\subsection{Case 3: $\babs{z-z_0}\le 8/\sqrt{m}$.} In this case,
$\e^{-\epsilon\sqrt{m}\min\{d,\babs{z_0-z}\}}\ge \e^{-8\epsilon}$, and thus
\eqref{spuc} implies
\begin{equation*}\babs{K_{m,n}(z_0,z)}^2\e^{-m(Q(z_0)+Q(z))}\le
C\e^{8\epsilon}m^2\e^{-\sqrt{m}\min\{d,\babs{z_0-z}\}}.
\end{equation*}
We have shown \eqref{bupp} in the case when $\babs{z_0-z}\le
8/\sqrt{m}$ with $C$ replaced by $C\e^{8\epsilon}$.

\subsection{Case 4: $m\ge m_0$, $md^2\ge 1$ and $\babs{z-z_0}\ge 8/\sqrt{m}$.}
In the sequel we fix any integer $n$ with $n\ge
](m-\apar)\tau+\bpar[$. Here $](m-\apar)\tau+\bpar[>0$ for all $m\ge
m_0$ by our choice of $m_0$ (see Subsect. \ref{nota}).

It is important to note that the assumption $md^2\ge 1$ means that
$1/\sqrt{m}\le d$ so that
\begin{equation*}8/\sqrt{m}\le
8d=\dist(z_0,\C\setminus(\setS_\tau\cap\setX)).\end{equation*} Our
starting point is Lemma \ref{lemm3}, which yields that
\begin{equation}\label{sim}\babs{K_{m,n}(z,z_0)}^2\e^{-mQ(z)}\le
Cm\int_{\D(z;1/\sqrt{m})}\babs{K_{m,n}(\zeta,z_0)}^2\e^{-mQ(\zeta)}\dA(\zeta),
\quad z\in \setS_\tau,
\end{equation}
where the number $C$ only depends on $A$. Define
\begin{equation*}\eps_0(\zeta)=\min\{\babs{z_0-\zeta}/2,4d\},\quad \zeta\in \C.\end{equation*}
Note that
\begin{equation}\label{siz}4/\sqrt{m}\le \eps_0(z)\le 4d.\end{equation}
Let $\chi_0$ be a smooth non-negative function such that
\begin{equation}\chi_0=0\quad \text{on}\quad \D(z_0;\eps_0(z)/2)\quad
\text{and}\quad \chi_0=1\quad\text{outside}\quad
\D(z_0;\eps_0(z)),\end{equation} and also $\babs{\dbar\chi_0}^2\le
(C/\eps_0(z)^2)\chi_0$ with $C$ an absolute constant ($C=5$ will
do). In view of \eqref{siz}, it yields that
\begin{equation*}\babs{\dbar\chi_0}^2\le
Cm\chi_0\quad \text{on}\quad \C,\end{equation*} where $C=5/16$.
Notice that $\dbar\chi_0$ is supported in the annulus
\begin{equation}\label{ann}U_0=U_0(z_0,z)=\{\zeta;\, \eps_0(z)/2\le \babs{z_0-\zeta}\le
\eps_0(z)\}.\end{equation}

Since $\chi_0$ is non-negative on $\C$ and since $\chi_0=1$ on
$\D(z;1/\sqrt{m})$, the estimate \eqref{sim} implies that
\begin{equation}\label{gops}\babs{K_{m,n}(z,z_0)}^2\e^{-mQ(z)}\le
Cm\int_\C\chi_0(\zeta)\babs{K_{m,n}(\zeta,z_0)}^2\e^{-mQ(\zeta)}\dA(\zeta),
\end{equation}
where $C$ only depends on $A$. Let $H_{\chi_0,m,n}$ be the linear
space $\Hspace_{m,n}$ with inner product
\begin{equation*}\langle f,g\rangle_{\chi_0,mQ}=\int_\C f\bar g\chi_0\e^{-mQ}
\dA.\end{equation*} We rewrite integral in the the right hand side
in \eqref{gops} in the following way
\begin{multline}
\int_\C
\chi_0(\zeta)\babs{K_{m,n}(\zeta,z_0)}^2\e^{-mQ(\zeta)}\dA(\zeta)
=\sup\biggl\{\babs{\langle
u,K_{m,n}(\cdot,z_0)\rangle_{\chi_0,mQ}}^2;\,
u\in\Hspace_{m,n},\, \int_\C \babs{u}^2\chi_0\e^{-mQ}\dA\le 1\biggr\}=\\
=\sup\biggl\{\babs{\bigl\langle\chi_0 u,
K_{m,n}(\cdot,z_0)\bigr\rangle_{mQ}}^2;\,u\in\Hspace_{m,n},
\,\int_\C\babs{u}^2\chi_0 \e^{-mQ}\dA\le 1\biggr\}=\\
=\sup\biggl\{\big|\Pop_{m,n}[\chi_0
u](z_0)\big|^2;\,u\in\Hspace_{m,n},\, \int_\C\babs{u}^2\chi_0
\e^{-mQ}\dA\le 1 \biggr\}, \label{basic-1}
\end{multline}
where $\Pop_{m,n}$ is the orthogonal projection of $L^2_{mQ}$ onto
$\Hspace_{m,n}$. Now fix $u\in\Hspace_{m,n}$ with
$\int_\C\babs{u}^2\chi_0\e^{-mQ}\dA\le 1$ and recall that
$\Pop_{m,n}[\chi_0 u]=\chi_0u-u_*,$ where $u_*$ is the
$L^2_{mQ,n}$-minimal solution to $\dbar u_*=\dbar(\chi_0
u)=u\dbar\chi_0$. In particular, $u_*$ is holomorphic in $\D(z_0;\eps_0(z)/2)$ and
$u_*=-\Pop_{m,n}[\chi_0 u]$ there. See Subsect. \ref{init}.

We intend to apply Th. \ref{boh} with a suitable real-valued
function $\Tfun_m$. We shall at first only specify $\Tfun_m$ by
requiring certain properties of it. An explicit construction of
$\Tfun_m$ is then given at the end of the proof.

\medskip
\noindent \bf Condition 1. \rm We require that
\begin{equation}\label{rc1}\Tfun_m=0\quad \text{on}\quad
\D(z_0;1/(2\sqrt{m})),\end{equation}

\medskip
\noindent \bf Condition 2. \rm There exists a number $\epsilon>0$
depending only on $a$, $\apar$ and $\tau$, such that
\begin{equation*}\Tfun_m(\zeta)\le
-\epsilon\sqrt{m}\eps_0(z)/4\quad \text{when}\quad \babs{\zeta-z_0}\ge
\eps_0(z)/2,\end{equation*}

\medskip
\noindent \bf Condition 3. \rm The various conditions on $\Tfun_m$
in Th. \ref{boh} are satisfied.
More precisely,

(i) $\Tfun_m$ is $\calC^{1,1}$-smooth and
\begin{equation*}(m-\apar)\Delta Q(\zeta)+\Delta \Tfun_m(\zeta)\ge
ma/2,\quad \text{for a.e.}\quad \zeta\in
\overline{\D}(z_0;6d),\end{equation*}

(ii) $\Tfun_m$ is constant in $\C\setminus \overline{\D}(z_0;6d)$,

(iii) We have that
\begin{equation*}\frac {|\dbar \Tfun_m|^2} {ma/2}\le \frac 1
{4\e^{\apar q_\tau}}\quad \text{on}\quad \C,\end{equation*}
where $q_\tau=\sup_{\setS_\tau}\{Q(\zeta)\}$.

It is clear that (i), (ii) and (iii) imply the conditions on
$\Tfun_m$ in Th. \ref{boh}, with $\kappa=1/2$.

We now turn to consequences of the conditions on $\Tfun_m$, and
start with condition 1. Applying Lemma \ref{lemm2}, we find that,
for any real function $\Tfun_m$ satisfying \eqref{rc1}, we have
\begin{equation}\label{gops4}\babs{u_*(z_0)}^2\e^{-mQ(z_0)}\le
Cm\int_{\D(z_0;1/(2\sqrt{m}))}\babs{u_*(\zeta)}^2\e^{-mQ(\zeta)}\dA(\zeta)\le
Cm\int_\C
\babs{u_*(\zeta)}^2\e^{\Tfun_m(\zeta)-mQ(\zeta)}\dA(\zeta),\end{equation}
where $C$ depends only on $A$.

By Condition 3, we may apply Th. \ref{boh} to the integral in the
right hand side in \eqref{gops4}. It yields that
\begin{equation}\label{gops5}\int_\C\babs{u_*}^2
\e^{\Tfun_m-mQ}\dA\le C\int_{U_0}\babs{u\dbar\chi_0}^2 \frac
{\e^{\Tfun_m-mQ}}{ma/2+\bpar c_\tau} \dA,
\end{equation}
with a number $C$ depending only on $\tau$ and $\apar$ (note that
$(1-\kappa)^{-2}=4$). Here
$c_\tau=\inf_{\setS_\tau}\{(1+\babs{\zeta}^2)^{-2}\}
\ge (1+R^2)^{-2}$ and $U_0$ is the annulus defined
in \eqref{ann}.

But since $\babs{\dbar\chi_0}^2\le Cm\chi_0$, \eqref{gops5} implies
\begin{equation*}\int_\C\babs{u_*}^2
\e^{\Tfun_m-mQ}\dA\le \frac {Cm}{am+\bpar
c_\tau}\int_{U_0}\babs{u}^2\chi_0\e^{\Tfun_m-mQ}\dA.\end{equation*}
Here $C$ only depends on $\apar$ and $\tau$. We now use Condition 2,
which implies that $\Tfun_m(\zeta)\le-\epsilon\sqrt{m}\eps_0(z)$
whenever $\zeta\in U_0$. It yields that we may continue to estimate
\begin{equation}\label{gops7}
\int_{U_0}\babs{u}^2\chi_0\e^{\Tfun_m-mQ}\dA\le
\e^{-\epsilon\sqrt{m}\eps_0(z)/4}\int_\C
\babs{u}^2\chi_0\e^{-mQ}\dA\le
\e^{-\epsilon\sqrt{m}\min\{\babs{z_0-z}/8,d\}},\end{equation} where
we have used that $\int_\C \babs{u}^2\chi_0\e^{-mQ}\dA\le 1$ in the
last step.

Tracing back through \eqref{gops}--\eqref{gops7}, we infer that
\begin{equation*}\babs{K_{m,n}(z_0,z)}^2\e^{-m(Q(z_0)+Q(z))}\le
C\frac {m^3} {am+\bpar
c_\tau}\e^{-\epsilon\sqrt{m}\min\{\babs{z_0-z}/8,d\}}\le
Ca^{-1}m^2\e^{-\epsilon\sqrt{m}\min\{\babs{z_0-z}/8,d\}},\end{equation*}
where $C$ depends on $\apar$, $\tau$, and $A$. This proves Lemma
\ref{propn1} (with $\epsilon/8$ instead of $\epsilon$ and $Ca^{-1}$
in place of $C$) under the hypotheses that a function $\Tfun_m$
satisfying conditions 1,2, and 3 above exists. To finish the proof
we must verify the existence of such a $\Tfun_m$.

\subsection{Construction of $\Tfun_m$.} We now construct a
function $\Tfun_m$ and a positive number $\epsilon$ which satisfy the conditions 1, 2, and 3 above.
We look for a radial function of the form
\begin{equation*}\Tfun_m(\zeta)=-\epsilon\sqrt{m}S_m(\babs{\zeta-z_0}),\end{equation*}
where the number $\epsilon>0$ will be fixed later. We start by
giving an explicit construction of $S_m$, the proof of conditions 1
though 3 will then be accomplished without difficulty.

We recall that $1/\sqrt{m}\le d$ and start by specifying the
derivative $S_m^\prime$ to be the piecewise linear continuous
function on $[0,\infty)$ such that
\begin{equation*}S_m^\prime=0\quad \text{on} \quad [0,1/(2\sqrt{m})]\cup
[6d,\infty)\quad \text{and}\quad S_m^\prime=1\quad \text{on}\quad
[1/\sqrt{m},5d],\end{equation*} and $S_m^\prime$ is affine on each
of the intervals $[1/(2\sqrt{m}),1/\sqrt{m}]$ and $[5d,6d]$. The
distributional derivative of $S_m^\prime$ is then a linear
combination of characteristic functions,
\begin{equation*}S_m^\dprime=2\sqrt{m}\1_{[1/(2\sqrt{m}),
1/\sqrt{m}]}-(1/d)\1_{[5d,6d]},\end{equation*}
so that (since $md^2\ge 1$)
\begin{equation*}\babs{S_m^\dprime}\le
\max\{2\sqrt{m},1/d\}=2\sqrt{m}.\end{equation*} We now define $S_m$
by requiring that $S_m(0)=0$. Since $S_m^\prime=0$ on
$[0,1/(2\sqrt{m})]$ it is then clear that $S_m=0$ on $[0,1/(2\sqrt{m})]$.
Moreover, when $2/\sqrt{m}\le t\le 5d$, we get
\begin{equation*}S_m(t)=\int_0^t S_m^\prime(s)\diff s\ge
t-1/\sqrt{m}\ge t/2,\quad t\in [2/\sqrt{m},5d],\end{equation*} since
$S_m^\prime=1$ on $[1/\sqrt{m},5d]$. When $t\ge 5d$, we plainly have
\begin{equation*}S_m(t)\ge 5d-1/\sqrt{m}\ge 4d.\end{equation*}
We conclude that
\begin{equation}\label{auto}S_m(t)\ge \min\{t/2,4d\},\quad t\ge
2/\sqrt{m}.\end{equation} In particular, denoting by $c_m$ the
constant value that $S_m$ assumes on $[6d,\infty)$, we have $c_m\ge
4d$. This finishes the construction of $S_m$, and the corresponding
function $\Tfun_m$ is clearly of class $\calC^{1,1}(\C)$.

We now verify the conditions 1 through 3 above.  First, Condition 1
is clear, since $S_m(t)=0$ when $t\le 1/(2\sqrt{m})$. Also, part (ii)
of condition 3 is clear; since $S_m(t)=c_m$ is constant when $t\ge
6d$, we have that $\Tfun_m(\zeta)=-\epsilon\sqrt{m}c_m$ is constant
when $\zeta\not\in \overline{\D}(0;6d)$.

Since $\eps_0(z)\ge 4/\sqrt{m}$, \eqref{auto} implies that
$\varrho_m(\zeta)=-\epsilon S_m(\babs{\zeta-z_0})\le-\epsilon\eps_0(\zeta)$ when
$\babs{\zeta-z_0}\ge\eps_0(z)/2$. Since moreover $\eps_0(\zeta)\ge \eps_0(z)/4$ in this case, we get that condition 2 is satisfied.

There remains to check parts (i) and (iii) of condition 3, and to make precise what we mean by "$\epsilon$". To this
end, we need the following estimates
\begin{equation}\label{827}\babs{\dbar
\Tfun_m(\zeta)}^2=\epsilon^2m\babs{S_m^\prime(\babs{\zeta-z_0})}^2/4\le
\epsilon^2 m/4,\qquad \zeta\in \C,\end{equation} and
\begin{equation}\label{828}\babs{\Delta \Tfun_m(\zeta)}\le \frac \epsilon 4\sqrt{m}\lpar
\babs{S_m^\dprime\lpar \babs{\zeta-z_0}\rpar}+\frac
{S_m^\prime\lpar\babs{z_0-\zeta}\rpar}
{\babs{z_0-\zeta}}\rpar\le \epsilon m,\qquad \zeta\in
\C,\end{equation} which follows immediately from the properties of
$S_m$ (since $\babs{z_0-\zeta}\ge 1/(2\sqrt{m})$ when
$S_m^\prime(\babs{z_0-\zeta})\ne 0$).

To verify (i), we use \eqref{828}. It yields that it suffices to
choose $\epsilon>0$ such that $(m-\apar)a-\epsilon m\ge ma/2$ for
$m\ge m_0$, i.e. $\epsilon\le (1/2-\apar/m)a.$ Since we have assumed
that $m_0\ge 4\apar$, it thus suffices to choose $\epsilon=a/4$.

We finally verify (iii). By \eqref{827}, it suffices to choose an
$\epsilon>0$ such that $(\epsilon^2/4)/(a/2)\le 1/(4\e^{\apar
q_\tau})$, i.e. $\epsilon^2 \le a/(2\e^{\apar q_\tau})$.

We have verified the existence of $\epsilon>0$, of the form
$\epsilon=\min\{\sqrt{a/(2\e^{\apar q_\tau})},a/2\}$. This shows
that the choice $\epsilon=c\min\{a,1\}$ works with a proportionality constant
$c$ which depends on $\apar$ and $q_\tau$. The proof is finished.
\hfill $\qed$

\subsection{The proof of Theorem \ref{th1.5}} Let $K$ be a
compact subset of $\setS_\tau^\circ\cap \setX$, and pick $M\ge 0$.
By Th. \ref{th3} we have that $K_{m,n}(z_0,z_0)\e^{-mQ(z_0)}=m\Delta
Q(z_0)+\calO(1)$ as $m\to\infty$ and $n\ge m\tau-M$, where the
$\calO$ is uniform for $z_0\in K$. It yields that
\begin{equation}\label{abo}\berd_{m,n}^{\langle z_0\rangle}(z)=\frac
{\babs{K_{m,n}(z,z_0)}^2} {K_{m,n}(z_0,z_0)}\e^{-mQ(z)}=\frac
{\babs{K_{m,n}(z,z_0)}^2} {m\Delta
Q(z_0)+\calO(1)}\e^{-m(Q(z)+Q(z_0))},\quad z\in\C,\end{equation} as
$m\to \infty$ and $n\ge m\tau-M$. Since $\Delta Q$ is bounded below
by a positive number on $K$, the right hand side in \eqref{abo} can
be estimated by
\begin{equation*}Cm^{-1}\babs{K_{m,n}(z,z_0)}^2\e^{-m(Q(z)+Q(z_0))},\quad z\in \C\end{equation*}
where $C$ depends on the lower bound of $\Delta Q$ on $K$. The
assertion now follows from Th. \ref{flock}.

\hfill$\qed$

\section{The Bargmann--Fock case and harmonic measure} \label{new}

\subsection{Preliminaries.} In this section we prove
Th. \ref{th5}. We therefore put $Q(z)=\babs{z}^2$. Recall that in
this case $\setS_\tau=\overline{\D}(0;\sqrt{\tau}),$ and (see
\eqref{bfock})
\begin{equation}
\label{bfock-2} \diff B^{\langle {z_0}\rangle}_{m,n}(z)=
m\frac{|E_{n-1}(mz\bar{z}_0)|^2}{E_{n-1}(m\babs{{z_0}}^2)}\e^{-m|z|^2}\diff
A(z)\qquad\text{where}\qquad E_k(z)=\sum_{j=0}^k\frac {z^j} {j!}.
\end{equation}

\subsection{The action on polynomials.}

\begin{prop} Fix a complex number ${z_0}\ne 0$, a positive integer $d$ and let $n$ be an integer, $n\ge d+1$.
Then, for all analytic polynomials
$u$ of degree at most $d$, we have
\begin{equation}
\label{pv} \pv\int_\C u(z^{-1})\diff B^{\langle {z_0}\rangle}_{m,n}
\to u(z_0^{-1}),\qquad\text{as}\quad m\to\infty,
\end{equation}
uniformly in $n$, $n\ge d+1$. \label{prop-pv}
\end{prop}

\begin{proof} It is sufficient to prove the statement for $u(z)=z^j$
with $j\le d$. The left hand side in \eqref{pv}
can then be written
\begin{equation*}\pv \int_\C z^{-j}\diff B^{\langle {z_0}\rangle}_{m,n}
=\frac{m\bfun_{m,n}^j({z_0})}{E_{n-1}(m\babs{{z_0}}^2)},\end{equation*}
where we have put
\begin{equation*}\bfun_{m,n}^j({z_0})=\pv \int_{\C}
z^{-j}\bigg|\sum_{k=0}^{n-1}\frac {(m{z_0}\bar{z})^k}
{k!}\bigg|^2\e^{-m\babs{z}^2}\dA(z).
\end{equation*}
Expanding the square yields
\begin{equation*}b_{m,n}^j({z_0})=\sum_{k,l=0}^{n-1}
\frac {m^{k+l}{z_0}^k\bar{z}_0^l}
{k!l!}\pv\int_{\C}\bar{z}^kz^{l-j}\e^{-m\babs{z}^2}\dA(z).
\end{equation*}
Clearly only those $k,l$ for which $k=l-j$ give a non-zero
contribution to the sum, and therefore,
\begin{equation*}b_{m,n}^j({z_0})={z_0}^{-j}\sum_{l=j}^{n-1}\frac
{m^{2l-j}\babs{{z_0}}^{2l}}
{(l-j)!l!}\int_{\C}\babs{z}^{2(l-j)}\e^{-m\babs{z}^2}\dA(z)=
\frac{1}{m{z_0}^{j}}\sum_{l=j}^{n-1}\frac {m^{l}\babs{{z_0}}^{2l}} {l!}.
\end{equation*}
It follows that
\begin{equation*}b_{m,n}^j({z_0})=\frac{1}{m{z_0}^{j}}\sum_{l=j}^{n-1}
\frac {(m\babs{{z_0}}^2)^l} {l!}=\frac{1}{m{z_0}^j}\lpar E_{n-1}(m\babs{{z_0}}^2)-
E_{j-1}(m\babs{{z_0}}^2)\rpar,\end{equation*}
and so
\begin{equation*}\frac {mb_{m,n}^j({z_0})} {E_{n-1}(m\babs{{z_0}}^2)}=\frac1{{z_0}^j}\lpar
1-\frac{E_{j-1}(m\babs{{z_0}}^2)}{E_{n-1}(m\babs{{z_0}}^2)}\rpar.
\end{equation*}
Finally, since $j\le d<n$,
$$\frac{E_{j-1}(m\babs{{z_0}}^2)}{E_{n-1}(m\babs{{z_0}}^2)}\le
\frac{E_{d-1}(m\babs{{z_0}}^2)}{E_{d}(m\babs{{z_0}}^2)}\to0\quad\text{as}\quad
m\to\infty.$$
\end{proof}

\begin{prop} \label{bulk}
Let $0<r<\sqrt{\tau}$, ${z_0}\in\C\setminus\overline{\D}(0;\sqrt{\tau})$
and $u$ an analytic polynomial. Then
\begin{equation*}\pv\int_{\D(0;r)}u(z^{-1})
\diff B^{\langle {z_0}\rangle}_{m,n}(z)\to0
\qquad\text{as}\quad
m\to\infty\quad \text{and}\quad  n/m\to\tau.
\end{equation*}
\end{prop}

\begin{proof} Put, for $\nu=0,1,2,\ldots$,
$$b_{m,n}^\nu({z_0},r)=\pv\int_{\D(0;r)}z^{-\nu}
\diff B^{\langle {z_0}\rangle}_{m,n}.$$ A straightforward calculation
based on \eqref{bfock-2} leads to
\begin{equation}
\label{grr}
b_{m,n}^\nu({z_0},r)=\frac 1
{z_0^\nu E_{n-1}(m\babs{{z_0}}^2)}\sum_{j=\nu}^{n-1}\frac {(m\babs{{z_0}}^2)^j}
{j!(j-\nu)!}\int_0^{mr^2} s^{j-\nu}\e^{-s}\diff s.
\end{equation}
We suppose $n$ is greater than $\nu$ by at least two units, so that we may
pick an integer $k$ with $\nu<k<n$, and split the sum \eqref{grr} accordingly:
\begin{equation}
\label{grr-1}
b_{m,n}^\nu({z_0},r)=\frac 1
{z_0^\nu E_{n-1}(m\babs{{z_0}}^2)}\Bigg\{\sum_{j=\nu}^{k-1}\frac {(m\babs{{z_0}}^2)^j}
{j!(j-\nu)!}\int_0^{mr^2} s^{j-\nu}\e^{-s}\diff s
+\sum_{j=k}^{n-1}\frac {(m\babs{{z_0}}^2)^j}
{j!(j-\nu)!}\int_0^{mr^2} s^{j-\nu}\e^{-s}\diff s\Bigg\}.
\end{equation}
We estimate the first term trivially as follows:
\begin{equation}
\label{grr-2} \sum_{j=\nu}^{k-1}\frac {(m\babs{{z_0}}^2)^j}
{j!(j-\nu)!}\int_0^{mr^2} s^{j-\nu}\e^{-s}\diff s\le
\sum_{j=\nu}^{k-1}\frac {(m\babs{{z_0}}^2)^j}
{j!(j-\nu)!}\int_0^{\infty} s^{j-\nu}\e^{-s}\diff s=\sum_{j=0}^{k-1}
\frac {(m\babs{{z_0}}^2)^j}{j!}=E_{k-1}(m|{z_0}|^2).
\end{equation}
As for the second term, we use the fact that the function $s\mapsto
s^{j-\nu}\e^{-s}$ is increasing on the interval $[0,j-\nu]$, to say
that \begin{equation*}\int_0^{mr^2} s^{j-\nu}\e^{-s}\diff
s\le(mr^2)^{j-\nu+1}\e^{-mr^2},\end{equation*} provided that $j\ge
mr^2+\nu$. It follows that if $k\ge mr^2+\nu$, then
\begin{equation*}\sum_{j=k}^{n-1}\frac {(m\babs{{z_0}}^2)^j}
{j!(j-\nu)!}\int_0^{mr^2} s^{j-\nu}\e^{-s}\diff s\le
(mr^2)^{1-\nu}\e^{-mr^2}\sum_{j=k}^{n-1}\frac
{(mr\babs{{z_0}})^{2j}} {j!(j-\nu)!}.\end{equation*} By Stirling's
formula, $j!\ge \sqrt{2\pi}j^{j+1/2}\e^{-j},$ so that
\begin{equation*}\frac {(mr\babs{{z_0}})^{2j}}
{j!}\e^{-mr^2}\le\frac{1}{\sqrt{2\pi j}}m^{j}|{z_0}|^{2j}
\lpar\frac{mr^2}{j}\e^{1-\frac{mr^2}{j}}\rpar^j.\end{equation*}
Since the function $x\mapsto x\e^{1-x}$ is increasing on the
interval $[0,1]$, it yields that \begin{equation*}\frac
{(mr\babs{{z_0}})^{2j}} {j!}\e^{-mr^2}\le\frac{1}{\sqrt{2\pi j}}
m^{j}|{z_0}|^{2j} \lpar\frac{mr^2}{k}\e^{1-\frac{mr^2}{k}}\rpar^{j},
\qquad mr^2+\nu\le k\le j.\end{equation*} We write
\begin{equation}
\label{eq-ckm} c_{k,m}=\frac{mr^2}{k}\e^{1-\frac{mr^2}{k}}\le1,
\qquad mr^2+\nu\le k,
\end{equation}
and conclude that
\begin{multline}
\sum_{j=k}^{n-1}\frac {(m\babs{{z_0}}^2)^j}
{j!(j-\nu)!}\int_0^{mr^2} s^{j-\nu}\e^{-s}\diff s\le
(mr^2)^{1-\nu}\sum_{j=k}^{n-1}
\frac{(m|{z_0}|^2c_{k,m})^j}{(j-\nu)!\sqrt{2\pi j}}\\
\le  \frac{(mr^2)^{1-\nu}}{\sqrt{2\pi}}
(c_{k,m})^k\sum_{j=k}^{n-1}\frac{(m|{z_0}|^2)^j}{(j-\nu)!}\le
\frac{mr^2}{\sqrt{2\pi}}\Big(\frac{|{z_0}|}{r}\Big)^{2\nu}
(c_{k,m})^kE_{n-\nu-1}(m|{z_0}|^2). \label{eq-101}
\end{multline}
Now, a combination \eqref{grr-2} and \eqref{eq-101} applied to \eqref{grr-1}
yields
\begin{multline}
\label{grr-1'} \babs{{z_0^{\nu}}b_{m,n}^\nu({z_0},r)}\le\frac
{E_{k-1}(m\babs{{z_0}}^2)}{E_{n-1}(m\babs{{z_0}}^2)}
+\frac{mr^2}{\sqrt{2\pi}}\Big(\frac{|{z_0}|}{r}\Big)^{2\nu}(c_{k,m})^k
\frac{E_{n-\nu-1}(m|{z_0}|^2)}{E_{n-1}(m|{z_0}|^2)}\\
\le\frac {E_{k-1}(m\babs{{z_0}}^2)}{E_{n-1}(m\babs{{z_0}}^2)}
+\frac{mr^2}{\sqrt{2\pi}}\Big(\frac{|{z_0}|}{r}\Big)^{2\nu}(c_{k,m})^k.
\end{multline}
We would like to show that each of the terms on the right hand side of
\eqref{grr-1'} can be made small by choosing $k$ cleverly.
As for the first term, we appeal to a theorem of Szeg\"o,
\cite{Sz}, Hilfssatz 1, p. 54, which states that
\begin{equation*}E_l(lx)=\frac 1 {\sqrt{2\pi l}}
(\e x)^l \frac x{x-1}\lpar 1+\eps_l(x)\rpar\qquad x>1,
\end{equation*}
where $\eps_l(x)\to 0$ uniformly on compact subintervals of
$(1,\infty)$ as $l\to\infty$. It follows that
\begin{equation*}
\frac{E_{k-1}(m\babs{{z_0}}^2)}{E_{n-1}(m\babs{{z_0}}^2)}=
\sqrt{\frac{n-1}{k-1}}\frac{m|{z_0}|^2-n+1}{m|{z_0}|^2-k+1}
\Big(\frac{\e m|{z_0}|^2}{k-1}\Big)^{k-1}\Big(\frac{\e
m|{z_0}|^2}{n-1}\Big)^{1-n}
\frac{1+\eps_{k-1}(\frac{m|{z_0}|^2}{k-1})}{1+\eps_{n-1}(\frac{m|{z_0}|^2}{n-1})}.
\end{equation*}
Finally, we decide to pick $k$ such that
$$k/m\to\beta$$
as $k,m\to\infty$, where $r^2<\beta<\tau$. We observe that with this
choice of $k$, the above epsilons tend to zero as $k,m,n\to\infty$.
The function
$$y\mapsto(\e/y)^{y},\qquad 0<y\le1$$
is is strictly increasing, so that with
$$y_1=\frac{k-1}{m|{z_0}|^2}\approx\frac{\beta}{|{z_0}|^2},
\quad y_2=\frac{n-1}{m|{z_0}|^2}\approx\frac{\tau}{|{z_0}|^2},$$
we have
$$\frac{(\e/y_1)^{y_1}}{(\e/y_2)^{y_2}}\le\theta<1,$$
where at least for large $k,m,n$, the number $\theta$ may be taken
to be independent of $k,m,n$. It follows that
$$\Big(\frac{\e m|{z_0}|^2}{k-1}\Big)^{k-1}\Big(\frac{\e m|{z_0}|^2}{n-1}\Big)^{1-n}
\le \theta^{-m|{z_0}|^2},$$
so that
\begin{equation*}
\frac{E_{k-1}(m\babs{{z_0}}^2)}{E_{n-1}(m\babs{{z_0}}^2)}\le(1+o(1))
\sqrt{\frac{\tau}{\beta}}\frac{|{z_0}|^2-\tau}{|{z_0}|^2-\beta}
\theta^{-m|{z_0}|^2}\to0
\end{equation*}
exponentially quickly as $k,m,n\to\infty$.

Finally, as for the second term, we observe that the numbers $c_{k,m}$
defined by \eqref{eq-ckm} have the property that
$$c_{k,m}\to \frac{r^2}{\beta}\e^{1-\frac{r^2}{\beta}}<1,$$
as $k,m,n\to\infty$ in the given fashion. In particular, the second
term converges exponentially quickly to $0$. The proof is complete.
\end{proof}

\begin{cor} \label{korp} Let ${z_0}\in\C\setminus\overline{\D}(0;\sqrt{\tau})$, and
let $\omega$ be an open set in $\C$ which contains the circle
$\T(0,\sqrt{\tau})$. Further, let $u$ be an analytic polynomial.
Then
\begin{equation*}\int_{\omega}u(z^{-1})\diff B^{\langle
{z_0}\rangle}_{m,n}(z)\to u(z_0^{-1})\qquad \text{as}\quad
m\to\infty\quad \text{and}\quad  n/m\to\tau.\end{equation*}
\end{cor}

\begin{proof} This follows from propositions \ref{bulk}
and \ref{prop1}.
\end{proof}

\subsection{The proof of Theorem \ref{th5}.}
Let ${\calH}_\tau$ be the class of continuous functions $\C^*\to\C$
which are harmonic on $\C^*\setminus\setS_\tau$. For a function
$f\in{\calC}_b(\C)$ we write $\wt{f}$ for the unique function of
class ${\calH}_\tau$ which coincides with $f$ on
$\D(0,\sqrt{\tau})$. We must show that \begin{equation*}\int_\C
f(z)\diff B^{\langle {z_0}\rangle}_{m,n}(z)\to \widetilde
f({z_0})\end{equation*} as $m\to\infty$ and $n/m\to\tau$. See e.g.
\cite{GM}, p. 90.

Convolving with the F\'ejer kernel, we see that $f$ may be uniformly
approximated by functions which on a neighborhood $\omega$ of
$\T(0,\sqrt{\tau})$ are of the form $u(z^{-1})$, with $u$ a harmonic
polynomial. We may therefore w.l.o.g. suppose that $f$ itself is of
this form, i.e. $f(z)=u(z^{-1})$ when $z\in \omega$.
 Thus $f(z)=\wt{f}(z)=u(z^{-1})$ on
$\omega$. By Prop. \ref{prop1}, \begin{equation*}\int_\C
(f(z)-\wt{f}(z)) \diff B^{\langle {z_0}\rangle}_{m,n}(z)\to
0,\end{equation*} as $m\to\infty$ and $n/m\to\tau$. Moreover, Cor.
\ref{korp} gives that
\begin{equation*}\int_\C \wt{f}(z)\diff B^{\langle
{z_0}\rangle}_{m,n}(z)=\int_\C u(z^{-1})\diff B^{\langle
{z_0}\rangle}_{m,n}(z)\to u(z_0^{-1})=\wt{f}({z_0}),\end{equation*}
as $m\to\infty$ and $n/m\to\tau$. \hfill $\qed$

\end{document}